\documentclass[11pt, letter]{article}
\usepackage{amsmath, amsfonts, amsthm, colonequals, fullpage, graphicx, hyperref, xcolor}
\usepackage[inline]{enumitem}

\title{Fluctuations of TASEP and LPP with general initial data}

\author{Ivan Corwin\thanks{Columbia University, Department of Mathematics,  2990 Broadway, New York, NY 10027, USA,
and Clay Mathematics Institute, 10 Memorial Blvd. Suite 902, Providence, RI 02903, USA, and Institute Henri Poincar\'{e}, 11 Rue Pierre et Marie Curie, 75005 Paris, France, and Massachusetts Institute of Technology, Department of Mathematics, 77 Massachusetts Avenue, Cambridge, MA 02139-4307, USA.\newline
email: \href{mailto:ivan.corwin@gmail.com}{\protect\nolinkurl{ivan.corwin@gmail.com}}},
Zhipeng Liu\thanks{Courant Institute of Mathematical Sciences, New York University,
251 Mercer Street, New York, NY 10012, USA\newline
  email: \href{mailto:zhipeng@cims.nyu.edu}{\protect\nolinkurl{zhipeng@cims.nyu.edu}}}, and
Dong Wang\thanks{Department of Mathematics, National University of Singapore, 10 Lower Kent Ridge Rd, Singapore 119076. \newline
  email: \href{mailto:matwd@nus.edu.sg}{\protect\nolinkurl{matwd@nus.edu.sg}}}}

\newtheorem{thm}{Theorem}[section]
\newtheorem{lem}[thm]{Lemma}
\newtheorem{cor}[thm]{Corollary}
\newtheorem{prop}[thm]{Proposition}

\newtheorem{definition}[thm]{Definition}
\newtheorem{claim}[thm]{Claim}
\theoremstyle{remark}

\newtheorem{rmk}{Remark}

\newcommand{\rood}[1]{}

\newcommand{\e}{\epsilon}
\newcommand{\Prob}{\mathbb{P}}
\newcommand{\Leq}{\mathcal{L}}
\newcommand{\Leqck}{\check{\mathcal{L}}}

\newcommand{\intZ}{\mathbb{Z}}

\newcommand{\realR}{\mathbb{R}}
\newcommand{\bigO}{\mathcal{O}}
\newcommand{\G}{\check{G}}

\newcommand{\B}{\mathbf{B}}
\newcommand{\A}{\mathbf{A}}
\newcommand{\AiryA}{\mathcal{A}}
\newcommand{\ie}{i.e.}
\newcommand{\eg}{e.g.}
\newcommand{\resp}{resp.}
\newcommand{\iid}{i.i.d.}
\newcommand{\anought}{\mathbf{a}_0}
\newcommand{\bnought}{\mathbf{b}_0}
\newcommand{\cnought}{\mathbf{c}_0}
\newcommand{\dnought}{\mathbf{d}_0}

\DeclareMathOperator{\GUE}{GUE}
\DeclareMathOperator{\GOE}{GOE}
\DeclareMathOperator{\ct}{central}
\DeclareMathOperator{\spiked}{spiked}

\DeclareMathOperator{\dist}{dist}

\DeclareMathOperator{\meso}{meso}
\DeclareMathOperator{\micro}{micro}
\DeclareMathOperator{\macro}{macro}
\DeclareMathOperator{\Ai}{Ai}
\DeclareMathOperator{\step}{step}
\DeclareMathOperator{\flatinit}{flat}
\DeclareMathOperator{\Bernoulli}{Bern}
\DeclareMathOperator{\wedgeflat}{step/flat}
\DeclareMathOperator{\wedgeBernoulli}{step/Bern}
\DeclareMathOperator{\flatBernoulli}{flat/Bern}
\DeclareMathOperator{\stat}{stat}
\DeclareMathOperator{\BM}{BM}
\DeclareMathOperator{\Tr}{Tr}

\begin{document}

\maketitle

\begin{abstract}
We prove Airy process variational formulas for the one-point probability distribution of (discrete time parallel update) TASEP with general initial data, as well as last passage percolation from a general down-right lattice path to a point. We also consider variants of last passage percolation with inhomogeneous parameter geometric weights and provide variational formulas of a similar nature. This proves one aspect of the conjectural description of the renormalization fixed point of the Kardar-Parisi-Zhang universality class.
\end{abstract}

\section{Introduction}

The totally asymmetric simple exclusion process (TASEP) is a prototypical interacting particle system, or (via integration) random growth process. The theory of hydrodynamics describes the law of large numbers for the evolution of the system's particle density, or height function. In particular, if $h(x;t)$ represents the height function, then $\e h(\e^{-1}x;\e^{-1}t)$ converges (as $\e\to 0$) as a space-time process to the deterministic solution to a Hamilton Jacobi equation with explicit (model dependent) flux (see, for example, \cite{Kriecherbauer-Krug10}). The solution, of course, depends on the initial data and in particular on the limit (as $\e\to 0$) of $\e h_0(\e^{-1}x)$. It is possible to consider initial data $h_{0,\e}$ which depends on $\e$ so that $\e h_{0,\e}(\e^{-1}x)$ has a limit.

The aim of the present paper is to describe, in a similar spirit, how fluctuations around the law of large numbers evolve. In particular, define
\begin{equation}\label{eq.he}
h^{\e}(x;t) = c_1 \e^{b} h(c_2 \e^{-1}x; c_3 \e^{-z} t) - \bar{h}^{\e}(x;t).
\end{equation}
Then it is conjectured in \cite{Corwin-Quastel-Remenik15} that if we take
\begin{equation}\label{eq.bz}
b=1/2, \quad\textrm{and}\quad  z=3/2,
\end{equation}
then for $c_1,c_2,c_3$ that are model dependent constants (chosen in terms of microscopic dynamics via the KPZ scaling theory \cite{Spohn14, Halpin_Healy-Krug-Meakin92}) and suitable centering $\bar{h}^{\e}(x;t)$ (coming from the hydrodynamic theory), the space-time process $h^{\e}(\cdot;\cdot)$ will have a universal limit $\mathfrak{h}(\cdot;\cdot)$ which is independent of the underlying model. The class of all models which satisfy this is called the Kardar-Parisi-Zhang universality class, and this limiting object is called the fixed point of this universality class.

Much of the description and almost all of the universality of this fixed point remains a matter of conjecture. One of the main conjectures provided in \cite{Corwin-Quastel-Remenik15} (see also the review \cite{Quastel-Remenik13}) about this fixed point is that its solution can be described via a variational problem (in the spirit of the Lax-Oleinik formula for the inviscid Burgers equation) involving a four-parameter random field called the space-time Airy sheet. A corollary of this conjectural description is that if the initial profile $h^{\e}(\cdot;0)$ converges (as a spatial process) to some function $\mathfrak{h}_0(\cdot)$, then we have the following distributional equality, valid for any fixed $x$:
\begin{equation}
  \Prob\big(\mathfrak{h}(x; 1)\geq -r\big)  = \Prob\Big(\max_{y\in \realR} \big(\AiryA(y) - (x-y)^2 -\mathfrak{h}_0(y)\big)\leq r\Big).
\end{equation}
Here $\AiryA(\cdot)$ is the Airy process (Section \ref{Sec.distributions}) and by scaling properties of $\mathfrak{h}$, this implies a similar conjecture for general $t$.

The main contribution of the present paper is a proof of this conjectured variational description for the limiting one-point distribution of discrete time parallel update TASEP.

Fix $q\in (0,1)$. Then the TASEP height function $h^{{\rm TASEP}}(x;t)$ is a continuous function which is composed of slope $\pm 1$ increments between each integer $x$ and $x+1$. The height function evolves so that independently, each $\vee$ (a $-1$ followed by $+1$ slope increment) present in the height function at integer time $t$ becomes a $\wedge$ (a $+1$ followed by $-1$ slope increment) at time $t+1$ with probability $1-q$. (See Section \ref{subsec:TASEP} for more on this process.) Define $h^{\e,{\rm TASEP}}(x;t)$ in the manner of (\ref{eq.he}), choose $b,z$ as in (\ref{eq.bz}), and fix the constants and take a particular $\bar{h}^{\e}(x; t)$
\begin{equation}
  \begin{gathered}
    c_1 = q^{-1/6} (1+\sqrt{q})^{-1/3},\quad c_2 = 2 q^{-1/6}(1+\sqrt{q})^{2/3},\quad c_3=2(1-\sqrt{q})^{-1}, \\
    \bar{h}^{\e}(x; t)=2q^{-1/6} (1+\sqrt{q})^{-1/3} t.
  \end{gathered}
\end{equation}

Theorem \ref{thm:main_TASEP} shows that if $h^{\e,{\rm TASEP}}(\cdot;0)$ converges in distribution (as a spatial process) to some function $\mathfrak{h}_0(\cdot)$ then (assuming certain growth hypotheses on $h^{\e,{\rm TASEP}}(x;0)$ as $x$ gets large)
$$
\lim_{\e\to 0} \Prob\Big(h^{\e,{\rm TASEP}}(x;1)\geq -r\Big)  = \Prob\Big(\max_{y\in \realR} \big(\AiryA(y) - (x-y)^2 -\mathfrak{h}_0(y)\big)\leq r\Big).
$$

Let us sketch where this result comes from (and how it is proved). Since the limiting result is phrased in terms of a variational problem, it is natural to look for a finite $\e$ variational problem. This is facilitated through the connection between discrete time parallel update TASEP and the geometric random weight last passage percolation (LPP) model. Recalling the description of the TASEP, for each $\vee$ whose bottom point is at position $(i,j)$, we can associated a random variable $w^{*}(i,j)$ which records the number of time steps until the $\vee$ becomes a $\wedge$. These $w^*(i,j)$ are i.i.d. and geometrically distributed with parameter $1-q$ so that $\Prob\big(w^* =k\big) = (1-q) q^{k - 1}$ for $k\in \{1,2,\ldots\}$. Let $G_{(x,y)}\big(h^{{\rm TASEP}}(\cdot;0)\big)$ denote the first time $t$ such that $h^{{\rm TASEP}}(x;t)>y$. The growth dynamics imply that $G_{(x,y)}\big(h^{{\rm TASEP}}(\cdot;0)\big)$ satisfies a simple recursion
\begin{equation}
G_{(x,y)}\big(h^{{\rm TASEP}}(\cdot;0)\big)= \max \Big(G_{(x-1,y-1)}\big(h^{{\rm TASEP}}(\cdot;0)\big),\,\,G_{(x+1,y-1)}\big(h^{{\rm TASEP}}(\cdot;0)\big)\Big) + w^*(x,y).
\end{equation}
Iterating this recursion yields the discrete variational formula
\begin{equation}
G_{(x,y)}\big(h^{{\rm TASEP}}(\cdot;0)\big) = \max_{\pi} \sum_{(i,j)\in \pi} w^*(i,j)
\end{equation}
where $\pi$ is any path starting at $(x,y)$ and proceeding by slope $\pm 1$ ($\swarrow$ or $\searrow$) increments of length $\sqrt{2}$ downward until it hits $h^{{\rm TASEP}}(\cdot;0)$. The sum over $(i,j)\in \pi$ is only over those integer vertices in $\pi$. This is illustrated in Figure \ref{fig:introfig}.

\begin{figure}[ht]
  \centering
  \includegraphics[scale=.9]{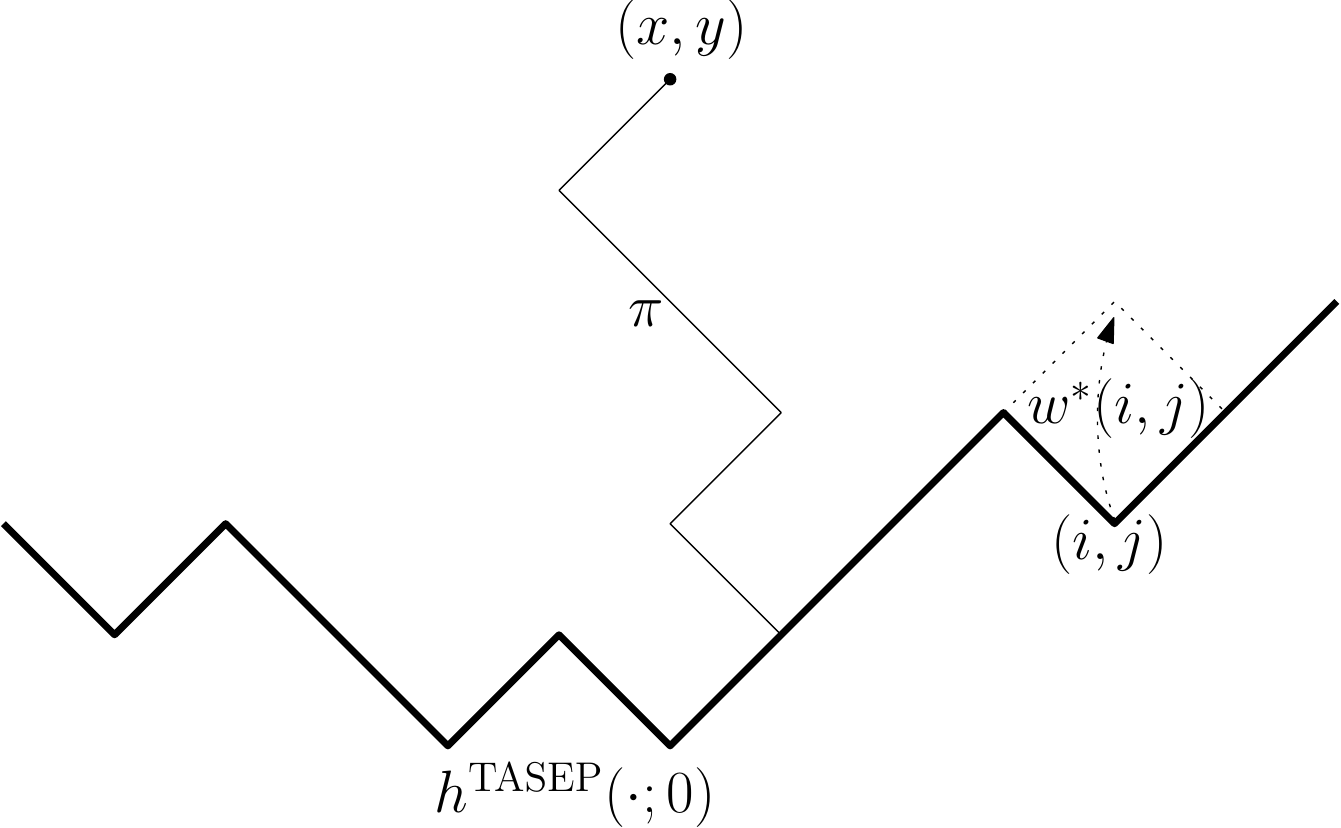}
  \caption{The discrete variational problem (last passage percolation) for TASEP.}
  \label{fig:introfig}
\end{figure}

Instead of restricting $\pi$ to end at any point of $h^{{\rm TASEP}}$, one can consider the maximal sum over paths which end at $(s,0)$, $s \in \intZ$. Johansson \cite{Johansson03} proved (see Proposition \ref{prop:Johansson_weak} below) that under the $b=1/2, z=3/2$ type scaling (i.e. up to constants replacing $s\mapsto \e^{-1}s$, $x\mapsto \e^{-1} x$, $y\mapsto \e^{-3/2} y$ and scaling the appropriately centered $G$ by $\e^{1/2}$) the $s$-indexed point-to-point last passage percolation converges as a process in $s$ to the Airy process $\mathcal{A}(\cdot)$ minus a parabola. However, along rays from $(x,y)$, the fluctuations of the last passage time enjoy slow decorrelation (see Theorem \ref{thm:uniform_slow_decorr}). Thus, as the endpoint of $\pi$ varies along $h^{{\rm TASEP}}(\cdot;0)$, the fluctuations remain Airy, up to a deterministic shift depending on $h^{{\rm TASEP}}(\cdot;0)$. It remains to show that the maximizer of the discrete variational problem converges to that of the limiting problem (i.e. that the resulting variational problem stays localized as $\e$ goes to zero). This requires certain growth conditions on $h^{{\rm TASEP}}(\cdot;0)$ and is achieved via a combination of large/moderate deviation bounds on TASEP and a utilization of Theorem \ref{lem:PNG_max} which contains some regularity estimates coming from the Gibbs property of the associated multi-layer PNG line ensemble (Section \ref{sec:Gibbs}).

In a similar manner we prove variational one-point distribution formulas for point to general curve LPP as well as LPP in which some of the weights have been perturbed. As a corollary of the TASEP and LPP results we provide variational formulas for a number of known one-point distributions, such as arise in TASEP with combinations of wedge, flat and Brownian-type initial data.

\subsection*{Organization of the paper}
Section \ref{sec:MMR} introduces the models (LPP and TASEP) as well as the main results (Theorems \ref{thm:main}, \ref{thm:main_TASEP}, and \ref{thm:inhomogeneous}) about them. The proofs of these theorems are applications of Theorem \ref{thm:uniform_slow_decorr} on the uniform slow decorrelation and Theorem \ref{lem:PNG_max} on the Gibbs property, and they are given in Section \ref{sec:proofs_of_main_theorems}. Proofs of  corollaries \ref{thm:Bernoulli_initial_condition} and \ref{cor:corrollaries_of_inhomogeneous_LPP} are given in Section \ref{sec:Proofs_of_corollaries}. The technical results, Theorems \ref{thm:uniform_slow_decorr} and \ref{lem:PNG_max}, are proved in Sections \ref{sec:uniform_slow_decor} and \ref{sec:Gibbs} respectively. Finally the appendix gives the proof of Lemma \ref{lem:weaker_technical}.

\subsection*{Acknowledgements}
The authors thank to Jinho Baik, Jeremy Quastel and Daniel Remenik for fruitful discussions. We also appreciate a close reading by our referees. Ivan Corwin was partially supported by the NSF through DMS-1208998 as well as by Microsoft Research and MIT through the Schramm Memorial Fellowship, by the Clay Mathematics Institute through the Clay Research Fellowship, by the Institute Henri Poincar\'{e} through the Poincar\'{e} Chair, and by the Packard Foundation through a Packard Foundation Fellowship. Zhipeng Liu greatfully acknowldeges the support from the department of mathematics, University of Michigan. Dong Wang was partially supported by the startup grant R-146-000-164-133.

\section{Models and main results}\label{sec:MMR}

\subsection{Point-to-curve LPP} \label{subsec:LPP_on_Z*Z}
Associate to each site $(i,j)\in \intZ^2$ an independent geometrically distributed random variable $w(i,j)$ with parameter $1 - q$, such that
\begin{equation} \label{eq:iid_geometric_weight}
  \Prob(w(i,j) = k) = (1 - q) q^k, \quad k = 0, 1, 2, \dotsc.
\end{equation}

The point-to-point last passage time between two lattice points $(x, y)$ and $(x', y')$ is denoted by $G_{(x', y')}(x, y)$ and defined by
\begin{equation} \label{eq:LPP_defn}
  G_{(x',y')}(x,y) \colonequals
  \begin{cases}
    \displaystyle \max_{\pi} \Big\{ \sum_{(i,j) \in \pi} w(i,j) \,\big\vert\, \pi \in (x, y) \nearrow (x', y') \Big\} & \text{if $x \leq x'$ and $y \leq y'$,} \\
    -\infty & \text{otherwise,}
  \end{cases}
\end{equation}
where $\pi$ stands for an up-right lattice path such that $\pi = (\pi_0 = (x, y), \pi_1, \pi_2, \dotsc, \pi_{x' + y' - x - y} = (x', y'))$ and $\pi_{k + 1} - \pi_k \in \{ (1, 0), (0, 1) \}$. More generally, if $(x', y')$ is a lattice point, and $(x, y)$ is on a line segment between two neighboring lattice points, then define
\begin{equation} \label{eq:non_integer_indices}
  G_{(x', y')}(x, y) := \text{the linear interpolation between}
  \begin{cases}
    G_{(x', y')}(x, [y]) \text{ and } G_{(x', y')}(x, [y] + 1) & \text{if $x \in \intZ$}, \\
    G_{(x', y')}([x], y) \text{ and } G_{(x', y')}([x] + 1, y) & \text{if $y \in \intZ$}.
  \end{cases}
\end{equation}
If $(x, y)$ and $(x', y')$ are lattice points, we define the short-handed notations for the reversed last passage time as
\begin{equation} \label{eq:defn_inverse}
  \G_{(x', y')}(x, y) = G_{(x, y)}(x', y') \quad \text{and} \quad \G(x, y) := \G_{(0, 0)}(x, y) = G_{(x, y)}(0, 0).
\end{equation}
We also define $\G_{(x', y')}(x, y)$ by linear interpolation if $(x, y)$ is on a line segment between two neighboring lattice points, analogous to \eqref{eq:non_integer_indices}. We will consider a more general point-to-curve last passage time, denoted by  $G_{(x', y')}(L)$ in this paper. Let $(x', y')$ be a lattice point and $L$ be a down-right lattice path in $\realR^2$ with $L = L(s) = \{ (x(s), y(s)) \mid s \in I \}$, for some interval $I\subset \realR$. Here a down-right lattice path means a (possibly infinite) directed path composed of down $\downarrow$ or right $\rightarrow$ oriented line segments which connect neighboring lattice points . 
Define
\begin{equation} \label{eq:defn_ptL}
  G_{(x', y')}(L) = \sup_{s \in I} \,\left\{ G_{(x', y')}(x(s), y(s)) \right\}.
\end{equation}
Although $s$ is a continuous parameter, it suffices to take the supremum among a discrete set of point-to-point last passage times.

As preliminaries for our work, let us recall some important results about the asymptotic  behavior of the point-to-point and point-to-curve last passage time. Focusing first on point-to-point last passage percolation, we state the law of large numbers, large/moderate deviations and the fluctuation limit theorems in the following proposition. Note that due to the symmetry of the lattice, we state our results in terms of $\G(x, y)$.

\begin{prop} \label{prop:previous_pt_to_pt}
  Fix $\gamma\in(0, \infty)$, then
  \begin{enumerate}[label=(\alph*)]
  \item (Johansson \cite{Johansson00}) \label{enu:prop:previous_pt_to_pt_1}
    \begin{equation} \label{eq:LLN_result}
      \lim_{N \to \infty} \frac{1}{N} \G(\gamma N, N) = a_0(\gamma), \quad\text{almost surely,}\quad \text{where} \quad a_0(\gamma) = \frac{(\gamma + 1) q + 2\sqrt{\gamma q}}{1 - q}.
    \end{equation}
  \item (Baik-Deift-McLaughlin-Miller-Zhou \cite{Baik-Deift-McLaughlin-Miller-Zhou01}) \label{enu:prop:previous_pt_to_pt_2}
    There exist a (large) constant $M > 0$ and a (small) constant $\delta > 0$ such that for large $N$, uniformly for all $M \leq x \leq \delta N^{1/3}$, there exists $c > 0$ such that
    \begin{equation}
      \Prob \left( \G(\gamma N, N) \leq a_0(\gamma)N - x N^{1/3} \right) \leq e^{-cx^3}. \label{eq:moderate_deviation_neg}
    \end{equation}
  \item (Johansson \cite{Johansson00}) \label{enu:prop:previous_pt_to_pt_3}
    \begin{equation}
      \lim_{N \to \infty} \Prob \left( \frac{\G(\gamma N, N) - a_0(\gamma) N}{b_0(\gamma) N^{1/3}} \leq x \right) = F_{\GUE}(x),
    \end{equation}
            where
    \begin{equation} \label{eq:fluctuation_constant_def}
      b_0(\gamma) = \frac{q^{1/6} \gamma^{-1/6}}{1 - q} (\sqrt{q} + \sqrt{\gamma})^{2/3} (1 + \sqrt{\gamma q})^{2/3},
    \end{equation}
    and $F_{\GUE}(x)$ is the GUE Tracy-Widom distribution for the limiting fluctuation of the largest eigenvalue in the Gaussian unitary ensemble (GUE), see Section \ref{Sec.distributions}.
  \end{enumerate}
\end{prop}
In this paper, we need a counterpart of \eqref{eq:moderate_deviation_neg}, which is stated  below, and proved in Appendix \ref{sec:appendix}.
\begin{lem} \label{lem:weaker_technical}
  Let $\gamma\in(0, \infty)$. There exist a (large) constant $M > 0$ and a (small) constant $\delta > 0$ such that for large $N$, uniformly for all $M \leq x \leq \delta N^{1/3}$, there exists $c > 0$ such that
  \begin{equation}
    \Prob \left( \G(\gamma N, N) \geq a_0(\gamma)N + x N^{1/3} \right) < e^{-cx}.
  \end{equation}
\end{lem}

Let us define a few constants which will be used throughout this paper:
\begin{equation}\label{eq:defn_nought}
\anought = a_0(1) = \frac{2\sqrt{q}}{1 - \sqrt{q}},\quad
\bnought = b_0(1) = \frac{q^{1/6} (1 + \sqrt{q})^{1/3}}{1 - \sqrt{q}}, \quad
\cnought = \frac{(1 + \sqrt{q})^{2/3}}{q^{1/6}},\qquad
\dnought = \frac{(1 + \sqrt{q})^{1/3}}{2q^{1/3}},
\end{equation}
as well as
\begin{equation}\label{eq:defn_noughtstar}
\anought^* = a^*_0(1) = \frac{2}{1 - \sqrt{q}} = \anought + 2,\qquad \dnought^*=\frac{q^{1/6}(1+\sqrt{q})^{1/3}}{2}.
\end{equation}

Define the limit shape curve (see Figure \ref{fig:L_and_L_check})
\begin{equation} \label{eq:defn_Lcheck}
  \begin{split}
    \Leqck := {}& \left\{ (x, y) \in (0, \infty) \times (0, \infty) \,\Big\vert\, y a_0 \left( \frac{x}{y} \right) = \anought \right\} \\
    = {}& \left\{ \big(r(\theta)\cos\theta, r(\theta)\sin\theta\big) \,\bigg\vert\, \text{$\theta \in (0, \frac{\pi}{2})$ and $r(\theta) = \frac{2(1 + \sqrt{q})}{(\cos\theta + \sin\theta) \sqrt{q} + 2\sqrt{\cos\theta \sin\theta}}$} \right\},
  \end{split}
\end{equation}
and
\begin{equation} \label{eq:defn_L}
  \Leq := \big\{ (x, y) \mid (1 - y, 1 - x) \in \Leqck \big\}.
\end{equation}
Note that
\begin{equation} \label{eq:c_3_lower_bound}
  \frac{\anought}{\lim_{x \to 0} a_0(x)} = 2 + 2q^{-1/2} \quad \text{and} \quad 1 - \frac{\anought}{\lim_{x \to 0} a_0(x)} = -1 - 2q^{-1/2},
\end{equation}
so $\Leqck$ ($\Leq$ \resp) is between $(2 + 2q^{-1/2}, 0)$ and $(0, 2 + 2q^{-1/2})$ ($(-1 - 2q^{-1/2}, 1)$ and $(1, -1 - 2q^{-1/2})$ \resp). Then by Proposition \ref{prop:previous_pt_to_pt}\ref{enu:prop:previous_pt_to_pt_1} we have that if $(x, y) \in \Leqck$, then with high probability, $\G([xN], [yN]) = \anought N + o(N)$, or equivalently, if $(x, y) \in \Leq$, then $G_{(N, N)}([xN], [yN]) = \anought N + o(N)$. In other words, these limit shapes reflect the law of large numbers under scaling by $N$ for those locations whose last passage time divided by $N$ is asymptotically $\anought$.
\begin{figure}[ht]
  \begin{minipage}[t]{0.45\linewidth}
    \centering
    \includegraphics{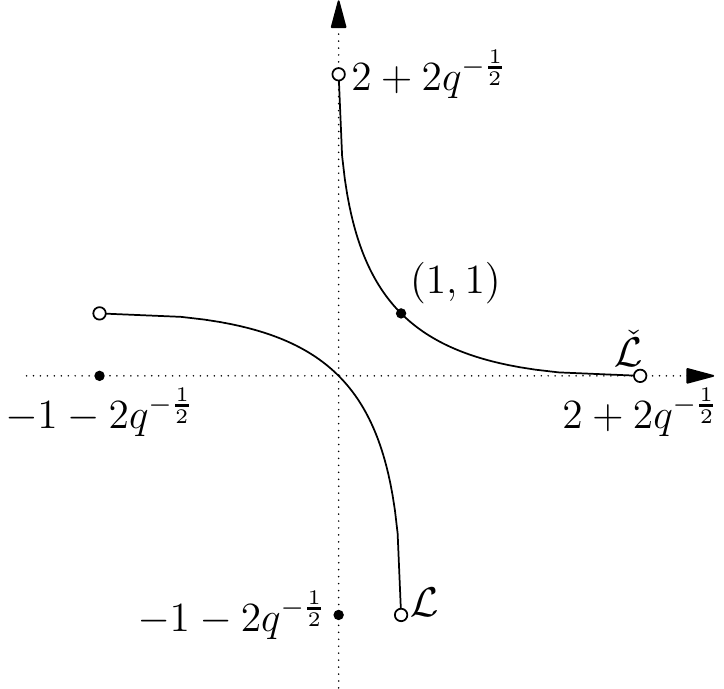}
    \caption{The shapes of $\Leq$ and $\Leqck$.}
    \label{fig:L_and_L_check}
  \end{minipage}
  \hspace{\stretch{1}}
  \begin{minipage}[t]{0.45\linewidth}
    \centering
    \includegraphics{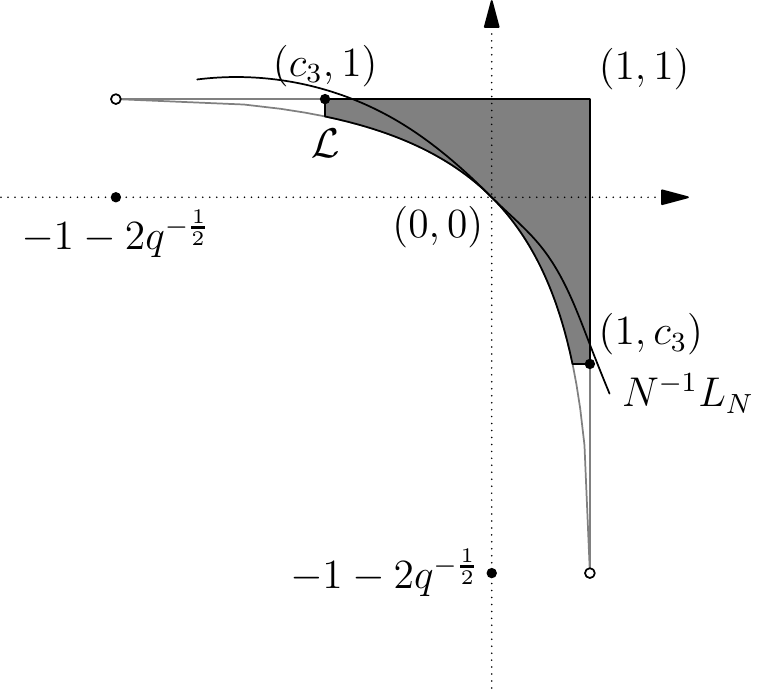}
    \caption{Regions $D$ (shaded) and $\tilde{D}$ ($D$ together with the two corners enclosed by gray lines) and an example of $N^{-1} L_N$.}
    \label{fig:L_N}
  \end{minipage}
\end{figure}


The following result shows how the Airy process $\AiryA(s)$ (see Section \ref{Sec.distributions}) arises in describing the spatial fluctuations of point-to-curve LPP. We define the saw-tooth curve $L^0$ that is approximately a anti-diagonal straight line
\begin{equation} \label{eq:defn_l^0}
  L^0 = \big\{ (-l^0(s) + s, -l^0(s) - s) \mid s \in \realR \big\}, \quad \text{where} \quad  l^0(s) =
  \begin{cases}
    k - s  & \text{if $s \in [k, k + \frac{1}{2}]$}, \\
    s - k - 1 & \text{if $s \in [k + \frac{1}{2}, k + 1]$}.
  \end{cases}
\end{equation}

\begin{prop}[Johansson \cite{Johansson03}] \label{prop:Johansson_weak}
  Define the stochastic process
  \begin{equation} \label{eq:H_N_def}
    H_N(s) \colonequals \frac{1}{\bnought N^{1/3}} \left( \G \left( N + l^0(\cnought N^{2/3} s) + s\cnought N^{2/3} , N + l^0(\cnought N^{2/3} s) - s\cnought N^{2/3} \right) - \anought N \right),
  \end{equation}
  where $\anought,\bnought,\cnought$ are defined in \eqref{eq:defn_nought}.
  Then on any interval $[-M, M]$, we have the weak convergence (as measures on $\mathcal{C}([-M,M],\realR)$) as $N \to \infty$ of
  \begin{equation}
    H_N(s) \Rightarrow \AiryA(s) - s^2.
  \end{equation}
\end{prop}
The definition and some properties of the Airy process are provided in Section \ref{Sec.distributions}. This functional limit theorem for the fluctuations of all $G_{(N, N)}(-l^0(s) + s, -l^0(s) - s)$ with $s = \bigO(N^{2/3})$, together with a tightness argument for large $s$, yields
\begin{prop}[Johansson \cite{Johansson03}] \label{prop:Johansson_max_Ai}
  As $N \to \infty$, the point-to-curve last passage time from $(N, N)$ to $L^0$ satisfies
  \begin{equation}
    \lim_{N \to \infty} \Prob \left( \frac{G_{(N, N)}(L^0) - \anought N}{\bnought N^{1/3}} \leq x \right) = \Prob \left( \max_{s \in \realR} (\AiryA(s) - s^2) \leq x \right).
  \end{equation}
\end{prop}

\subsection{Main result on fluctuations in point-to-curve LPP} \label{subsec_main_result_LPP}

Our main result, Theorem \ref{thm:main}, provides a similar variational characterization as Johansson's result (Proposition \ref{prop:Johansson_max_Ai}) for point-to-curve LPP with a general class of the down-right lattice paths. Before stating our theorem, we specify the class of down-right lattice paths which we will consider.

We will consider the point-to-curve LPP where the point is $(N,N)$ (or more generally $(N + [\sigma \cnought N^{2/3}], N - [\sigma \cnought N^{2/3}])$ for some $\sigma$) and the curve is a down-right lattice path $L_N$. As suggested by the subscript $N$, we will allow the curve to vary with $N$. However, to have a meaningful result, $L_N$ must satisfy two main properties. The first part of the hypothesis we impose is that under the scaling in which a window (which we call the central part) around $(0,0)$, of size $\bigO(N^{2/3})$ in the anti-diagonal direction and $\bigO(N^{1/3})$ in the diagonal direction, becomes of unit order, $L_N$ should converge to a function which we will denote by $\ell(\cdot)$. In fact, we will start with $\ell(\cdot)$ specified and define $L_N$ based on it. We will allow some variation from $\ell(s)$ which is denoted by $l_N(s)$, but will assume that in the window $|l_N(s)|$ is bounded by a sequence $m_N$ which goes to zero. The second part of the hypothesis ensures that outside the central part, $L_N$ should be sufficiently bounded away from the limit-shape $\Leq$. The purpose of this is to ensure that the maximizing path localizes in the central part. In fact, in order to ensure this localization we also assume that $\ell(s)$ does not grow any faster than $c_1s^2$ for some $c_1<1$. This is because the limit shape of $\Leq$ defined in \eqref{eq:defn_L} looks (with the scaling parameters we are using) like $s^2$ in the vicinity of the origin. (We also assume this for $|\ell(s)|$ for technical reason, see Remark \ref{rmk:lower_bound_ell}.)

We name the below hypothesis $\mathrm{Hyp}\big(C,c_1,c_2,c_3,a_{\infty},b_{\infty},\{m_N\}\big)$ owing to its dependence on certain parameters and functions. Owing to a coupling between LPP and TASEP, we also make use of a slightly modified hypotheses which we name $\mathrm{Hyp}^{*}\big(C,c_1,c_2,c_3,a_{\infty},b_{\infty},\{m_N\}\big)$ and describe near the end of the following definition.

\begin{definition}\label{thmhypo}
Consider constants
\begin{equation}
C>0,\quad c_1\in (0,1),\quad c_2\in (0,1/3),\quad c_3\in (-1 - 2q^{-1/2}, 0),\quad a_{\infty} \in \{ -\infty \} \cup \realR, \quad b_{\infty} \in \{ +\infty \} \cup \realR,
\end{equation}
and a sequence $\{m_N\}_{N\geq 1}\subset \realR_{+}$ converging to zero as $N$ goes to infinity. From \eqref{eq:defn_Lcheck}, \eqref{eq:defn_L} and Figure \ref{fig:L_and_L_check} it is clear that the horizontal line $y = 1$, the vertical line $x = 1$ and the curve $\Leq$ enclose a region, which we denote by $\tilde{D}$. Given $c_3$, define the region $D$ as the main part of $\tilde{D}$ with the two sharp corners cut off (Figure \ref{fig:L_N}):
\begin{equation} \label{eq:defn_D_D_tilde}
  D = \tilde{D} \setminus \{ (x, y) \mid \text{$x < c_3$ or $y < c_3$} \}.
\end{equation}
The lower bound on $c_3$ of  $-1 - 2q^{-1/2}$ is shown in \eqref{eq:c_3_lower_bound} and Figure \ref{fig:L_N} to correspond with the corner of $D$. Given $c_2$, for a down-right lattice path define its central part as
\begin{equation} \label{eq:defn_L^ct}
  L^{\ct} = \left\{ (x, y) \in L \,\big\vert \,\lvert x - y \rvert < 2 \cnought N^{2/3 + c_2}\right\}.
\end{equation}

We say that a continuous function $\ell:\realR\to \realR$ and a sequence of down-right lattice paths $L_N$ satisfy $\mathrm{Hyp}\big(C,c_1,c_2,c_3,a_{\infty},b_{\infty},\{m_N\}\big)$ if the following properties hold:
\begin{enumerate}
\item The function $\ell(s)$ satisfies the bound for $s\in \realR$ that
\begin{equation} \label{eq:upper_bound_l}
  |\ell(s)| < C + c_1 s^2.
\end{equation}
\item There is a sequence of intervals $I_N = (a_N, b_N) \subset (-N^{c_2},N^{c_2})$ converging to $(a_{\infty},b_{\infty})$ such that
    \begin{equation} \label{eq:L^ct}
      L^{\ct}_N = \left\{ \left( s\cnought N^{2/3} - (\ell(s) + l_N(s)) \dnought N^{1/3}, -s\cnought N^{2/3} - (\ell(s) + l_N(s)) \dnought N^{1/3} \right)\, \bigg\vert\, s \in I_N \right\},
    \end{equation}
    where $l_N(s): I_N \to \realR$ is some continuous function with $\max_{s \in I_N}\  \lvert l_N(s) \rvert \leq m_N$, and $\lim_{N \to \infty} m_N = 0$,
    and $\dnought$ is defined in \eqref{eq:defn_nought}.
\item The non-central part of $L_N$ satisfies
    \begin{equation} \label{eq:region_D}
      \left\{ (x, y) \,\big\vert\, (Nx, Ny) \in L_N \setminus L^{\ct}_N \right\} \cap \big((-\infty, 1] \times (-\infty, 1]\big) \subseteq D,
    \end{equation}
    as depicted in Figure \ref{fig:L_N}, and
    \begin{equation} \label{eq:L_N_outside}
      \dist \bigg( \big(L_N \setminus L^{\ct}_N\big) \cap \big((-\infty, N] \times (-\infty, N]\big) , \Big\{ (x, y) \,\big\vert\, \big( \tfrac{x}{N}, \tfrac{y}{N} \big) \in \Leq \Big\} \bigg) > N^{1/3 + 2c_2},
    \end{equation}
    where the distance is Euclidean. We have no requirement of $L_N$ outside of the region $(-\infty, N] \times (-\infty, N]$, because
\begin{equation} \label{eq:L_N_intersect_sector}
  G_{(N, N)}(L_N) = G_{(N, N)}\big(L_N \cap (-\infty, N] \times (-\infty, N]\big).
\end{equation}
\end{enumerate}

Let us also define a second (quite similar) hypothesis. Assume $C,c_1,c_2,a_{\infty},b_{\infty},\{m_N\}$ are as above and replace $c_3\in (-1 - 2q^{-1/2},0)$ by $c_3^*\in (-1 - 2q^{1/2},0)$.
We say that $\ell$ and the sequence $L_N^*$ satisfy $\mathrm{Hyp}^*\big(C,c_1,c_2,c_3^*,a_{\infty},b_{\infty},\{m_N\}\big)$ if they satisfy the above properties, with $\anought, \dnought, \Leq, D$ replaced by $\anought^*, \dnought^*, \Leq^*, D^*$. The first two starred terms $\anought^*, \dnought^*$ are given in Definition \eqref{eq:defn_noughtstar}, whereas $\Leq^*$ is define in \eqref{eq:Leqstar} and $D^*$ is defined as follows. The horizontal line $y = 1$, the vertical line $x = 1$ and the curve $\Leq^*$ enclose a region $\tilde{D}^*$. Given $c_3^*$, define the region
\begin{equation}
  D^* = \tilde{D}^* \setminus \{ (x, y) \mid \text{$x < c_3^*$ or $y < c_3^*$} \}.
\end{equation}
\end{definition}

We are interested in the limiting fluctuations of $G_{(N, N)}(L_N)$. The result which we now state describes that in terms of a variational problem involing the function $\ell(\cdot)$ related to $L_N$ through the above hypothesis. We actually state our result with an extra parameter $\sigma\in \realR$ which corresponds to shifting the point $(N,N)$ along an anti-diagonal line.
\begin{thm} \label{thm:main}
Fix constants $C, c_1, c_2, c_3, a_{\infty}, b_{\infty}$, a sequence $\{m_N\}_{N\geq 1}$ as in Definition \ref{thmhypo}, and $\sigma \in \realR$. Then for all $\epsilon>0$ there exists $N_0$ such that for all continuous functions $\ell:\realR\to \realR$ and sequences of down-right lattice paths $L_N$ satisfying $\mathrm{Hyp}\big(C,c_1,c_2,c_3,a_{\infty},b_{\infty},\{m_N\}\big)$, and for all $N>N_0$ and $x\in \realR$,
  \begin{equation} \label{ea:ineq_main}
    \left\lvert \Prob \bigg( \frac{G_{(N + [\sigma \cnought N^{2/3}], N - [\sigma \cnought N^{2/3}])}(L_N) - \anought N}{\bnought N^{1/3}} \leq x \bigg) - \Prob \bigg( \max_{s \in (a_{\infty}, b_{\infty})} \left( \AiryA(s) - (s - \sigma)^2 + \ell(s) \right) \leq x \bigg) \right\rvert
    < \epsilon.
  \end{equation}
\end{thm}


\begin{rmk}
First, note that $\max_{s \in (a_{\infty}, b_{\infty})} \left(\AiryA(s) - (s - \sigma)^2 + \ell(s)\right)$ is a well defined random variable, see Corollary \ref{cor:well_defined_ness}. Second, observe that Theorem \ref{thm:main} (as well as the subsequently stated results of Theorems \ref{thm:main_TASEP} and \ref{thm:inhomogeneous}) are stated in terms of a deterministic down-right lattice path / initial condition / boundary data. Here in the statement of the theorem, and later in the proof, we show that the lower bound of convergence rate is independent of the particular form of $\ell(s)$ and $L_N$ so long as they satisfy the hypothesis. (The bound of rate, however, does depend on the parameters of the hypothesis.) This independence will be used later to deal with the case of random $\ell(s)$, say, distributed as the path of random walk. Let us briefly see how this works. Assume that $\ell(s)$ is random and that for all $\e>0$ there exist constants $C\in \realR$ and $c_1\in (0,1)$ such that with probability at least $1-\e$, $|\ell(s)|<C+c_1 s^2$ for all $s$. Then Theorem \ref{thm:main} holds for such a random $\ell(s)$. Instead of coupling all initial data $L_N$ to a single (possibly random) $\ell(s)$ it is also possible to consider $L_N$ which satisfy all of the conditions of Hypothesis \ref{thmhypo}, except that \eqref{eq:L^ct} is replaced by
    \begin{equation}
      L^{\ct}_N = \left\{ \left( s\cnought N^{2/3} - (\ell_N(s) + l_N(s)) \dnought N^{1/3}, -s\cnought N^{2/3} - (\ell_N(s) + l_N(s)) \dnought N^{1/3} \right)\, \bigg\vert\, s \in I_N \right\},
    \end{equation}
where $\ell_N(s)$ converges as a spatial process to some (possibly random) $\ell(s)$ satisfying the aforementioned bounds.
\end{rmk}

\begin{rmk} \label{rmk:lower_bound_ell}
  It is essential that $\ell(s)$ is bounded above quadratically for the right-hand side of \eqref{ea:ineq_main} to make sense, but the lower bound of $\ell(s)$ is unimportant there. Actually we assume the lower bound of $\ell(s)$ in \eqref{eq:upper_bound_l} only for technical reason, in \eg\ \eqref{eq:micro_estimate1}.
\end{rmk}

Theorem \ref{thm:main} is proved in Section \ref{sec:proof_thm:main}. There are two main ingredient in the proof. The first ingredient is to show that the end of the longest path localizes to those $(x,y)$ in the vicinity of $(0, 0)$ so that $x + y = \bigO(N^{1/3})$ and $x - y = \bigO(N^{2/3})$. This is achieved using moderate deviation bounds along with certain regularity results in Theorem \ref{lem:PNG_max} which follow from the Gibbs property associated with this model. Having localized our consideration, the second ingredient is to show that the theorem holds when $L_N$ is restricted to a region around $(0,0)$ of length $\bigO(N^{2/3})$. This is achieved by the uniform slow decorrelation property of the LPP model, see Theorem \ref{thm:uniform_slow_decorr}.

\subsection{TASEP with general initial data}\label{subsec:TASEP}

For the analysis of the TASEP model, we introduce a slightly different LPP model where the \iid\ random variables $w^*(i, j)$ associated to each site are geometrically distributed on $\intZ_{>0}$
\begin{equation} \label{eq:dist_star_def}
  w^*(i, j) \sim w(i, j) + 1, \quad \text{such that} \quad \Prob(w^*(i,j)=k)=(1-q)q^{k-1}, k=1,2,\dotsc.
\end{equation}
We similarly define the point-to-point LPP $G^*_{(x',y')}(x,y)$, point-to-curve LPP $G^*_{(x',y')}(L)$, and the reversed LPP $\G^*_{(x', y')}(x, y)$ by \eqref{eq:non_integer_indices}, \eqref{eq:defn_ptL} and \eqref{eq:defn_inverse} with the weights changed from $w(i, j)$ to $w^*(i, j)$. They have simple relations to the LPPs $G^*_{(x',y')}(x,y)$, $G^*_{(x',y')}(L)$ and $\G^*_{(x', y')}(x, y)$ defined there, for example, if one of $(x,y)$ and $(x', y')$ is a lattice point,
\begin{equation}
  G^*_{(x', y')}(x, y) = G_{(x', y')}(x, y) + x' + y' - x - y + 1.
\end{equation}

The TASEP model considered in our paper is that with discrete time and parallel updating dynamics \cite{Borodin-Ferrari-Sasamoto08}, and is defined as follows. Let infinitely many particles be initially at time $t = 0$ placed on the integer lattice $\intZ$ such that no lattice site is occupied by more than one particle, and there are infinitely many particles to the left of $0$. 
At each integer time, the particles decide whether to jump to the right neighboring site simultaneously. For any particle $x$ at time $t$, if its right neighboring site $x(t) + 1$ is occupied then it does not move and $x(t + 1) = x(t)$; otherwise it jumps to the right neighboring site ($x(t + 1) = x(t) + 1$) with  probability $1-q$, or does not move ($x(t + 1) = x(t)$) with probability $q$.
\begin{figure}[ht]
  \centering
  \includegraphics{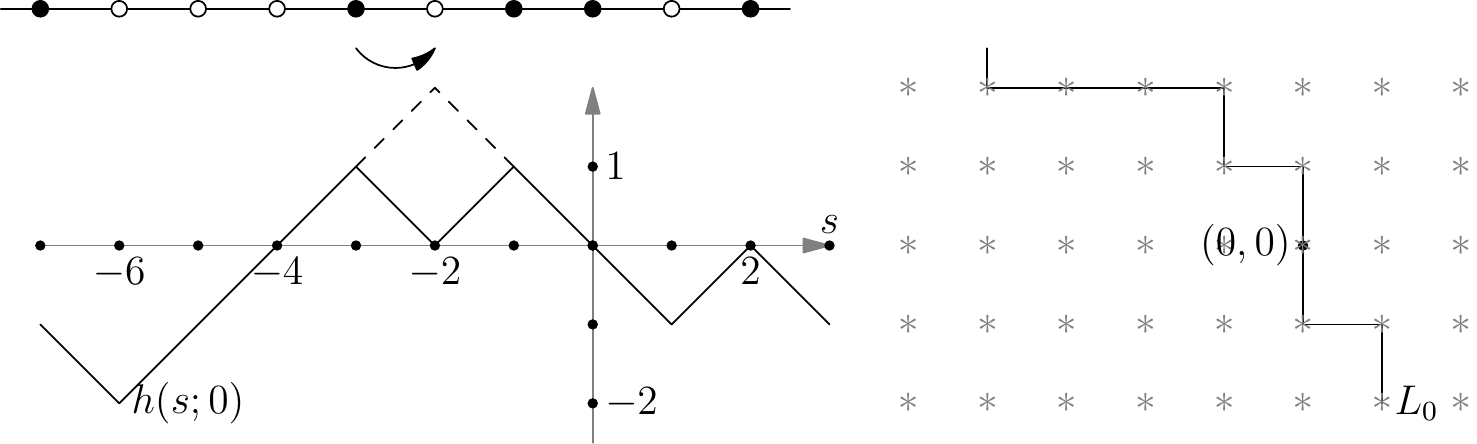}
  \caption{The height function $h(s; t)$ at the initial time $t = 0$. If at time $t = 1$ one particle jumps from $-3$ to $-2$, then $h(s; 1)$ is changed into the dashed shape. The polygonal chain $L$ representing the initial state of the model is shown on the right.}
  \label{fig:TASEP_and_h}
\end{figure}

At any time $t \geq 0$, we represent the positions of the particles by the height function $h(\cdot; t): \realR \to \realR$. We let $h(0; t) = 2N_t$ where $N_t$ is the number of particles that have jumped from site $-1$ to site $0$ during the time interval $[0, t)$. For any integer $k$, we define $h(k; t)$ inductively from $h(0; t)$ by $h(k + 1; t) - h(k; t) = \pm 1$ where the sign is positive (negative \resp) if the site $k$ is vacant (occupied \resp) by a particle at time $t$. At last, for non-integer $s$, we define $h(s; t)$ by the linear interpolation between $h([s]; t)$ and $h([s] + 1; t)$. See Figure \ref{fig:TASEP_and_h} for an example. Also noting that $h(s; t) = h(s; [t])$ for all $t \in \realR_+$, we have that $h(s; t)$ is determined by the values of $h(k; n)$ where $k, n \in \intZ$. Another observation is that the value of $h(k; t)$ is an integer that has the same parity of $k$. In the introduction we provided an alternative description for the dynamics of TASEP acting on the height function by changing $\vee$ to $\wedge$ according to geometrically distributed waiting times.

To analyze the dynamics of the TASEP model, or equivalently, the dynamics of the height function $h(s; t)$, we introduce the polygonal chain
\begin{equation} \label{eq:L_by_h}
  L = \left\{ \left( \frac{s}{2} + \frac{1}{2}h(s; 0), -\frac{s}{2} +\frac{1}{2} h(s; 0) \right) \bigg\vert\, s \in [K_1, K_2] \right\},
\end{equation}
to represent the initial configuration of the model, as shown in Figure \ref{fig:TASEP_and_h}, where  $K_1$ is the position of the leftmost unoccupied site at $t = 0$ if it exists, or $-\infty$ otherwise, and $K_2$ is one plus the position of the rightmost occupied site at $t = 0$ if it exists, or $+\infty$ otherwise.

As described in the introduction, the TASEP model can be coupled to the LPP model with weights $w^*(i, j)$ defined in \eqref{eq:dist_star_def} (see also \cite{Johansson00, Corwin-Ferrari-Peche12} for example). The relation between the distribution of $h(j; t)$ and the LPP is given by
\begin{equation} \label{eq:relation_TASEP_LPP}
  \Prob\big(h(j; t) > k\big) = \Prob \Big( G^*_{(\frac{k + j}{2}, \frac{k - j}{2})}(L) \leq t \Big),
\end{equation}
for any $j,k \in \intZ$ with the same parity.
Here $L$ is the polygonal chain defined in \eqref{eq:L_by_h}. This coupling follows by defining the TASEP height function at time $t$ as the rotated envelop of all points which has last passage time less than or equal to $t$. The weights correspond with the probabilities of particle movement.

\subsection{Main result on TASEP with general initial data}

Now we consider the TASEP model with general initial condition. Since the TASEP model is mapped to the LPP model with weight function given in \eqref{eq:dist_star_def}, the result for the TASEP is analogous to that of the LPP model stated in Section \ref{subsec_main_result_LPP}. Below we set up the notations for the LPP model with weight \eqref{eq:dist_star_def}, give technical conditions in terms of LPP, and then present the result in terms of the TASEP model.

Analogous to Proposition \ref{prop:previous_pt_to_pt}\ref{enu:prop:previous_pt_to_pt_1}, we have
\begin{equation}
  \lim_{N \to \infty} \frac{1}{N} \G^*(\gamma N, N) = a^*_0(\gamma) \quad \text{almost surely,}\quad \text{where} \quad a^*_0(\gamma) = a_0(\gamma) +\gamma+1 = \frac{\gamma+1 + 2\sqrt{\gamma q} }{1 - q}.
\end{equation}
Likewise analogous to $\Leqck$ and $\Leq$ defined in \eqref{eq:defn_Lcheck} and \eqref{eq:defn_L} (and recalling $\anought^* = a^*_0(1)$ from (\ref{eq:defn_noughtstar})), we define
\begin{equation}
  \begin{split}
    \Leqck^* := {}& \left\{ (x, y) \in (0, \infty) \times (0, \infty) \, \big\vert\,  y a^*_0 \left( \frac{x}{y} \right) = \anought^* \right\} \\
    = {}& \bigg\{ (r(\theta)\cos\theta, r(\theta)\sin\theta)\, \big\vert\, \text{$\theta \in (0, \frac{\pi}{2})$ and $r(\theta) = \frac{2(1 + \sqrt{q})}{(\cos\theta + \sin\theta) + 2\sqrt{\cos\theta \sin\theta q}}$} \bigg\},
  \end{split}
\end{equation}
and
\begin{equation}\label{eq:Leqstar}
  \Leq^* := \{ (x, y) \mid (1 - y, 1 - x) \in \Leqck^* \}.
\end{equation}
By Theorem \ref{prop:previous_pt_to_pt}\ref{enu:prop:previous_pt_to_pt_1}, \ref{enu:prop:previous_pt_to_pt_2}, we have that if $(x, y) \in \Leqck^*$, $\G^*(xN, yN) = \anought^*N + o(N)$, and equivalently if $(x, y) \in \Leq^*$, $G^*_{(N, N)}(xN, yN) = \anought^*N + o(N)$, with high probability.

\begin{thm} \label{thm:main_TASEP}
Fix constants $C, c_1, c_2, c^*_3, a_{\infty}, b_{\infty}$, a sequence $\{m_N\}_{N\geq 1}$ as in Definition \ref{thmhypo}, and $\sigma \in \realR$. We will consider down-right lattice paths $L^*_N$ defined by the initial condition of a TASEP model, that is, for each index $N$, we consider the TASEP model represented by a height function $h_N(s; t)$, and then let $L^*_N$ be the polygonal chain $L$ that is defined by $h_N(s; 0)$ in \eqref{eq:L_by_h}. Then for all $\epsilon>0$ there exists $N_0$ such that for all continuous functions $\ell:\realR\to \realR$ and sequences of down-right lattice paths $L^*_N$ satisfying $\mathrm{Hyp}^*\big(C,c_1,c_2,c^*_3,a_{\infty},b_{\infty},\{m_N\}\big)$, and for all $N>N_0$ and $x\in \realR$,
\begin{equation}
    \left\lvert \Prob \bigg( \frac{h_N(2\sigma \cnought N^{\frac{2}{3}}; \anought^* N) - 2N}{2\dnought^* N^{\frac{1}{3}}} > -x \bigg)  - \Prob \bigg( \max_{s \in (a_\infty, b_\infty)} \left( \AiryA(s) - (s - \sigma)^2 + \ell(s) \right) < x \bigg) \right\rvert < \epsilon.
\end{equation}
  \end{thm}
This result is proved in Section \ref{sec:proof_thm:TASEP} as a straight-forward consequence of Theorem \ref{thm:main}.

\begin{rmk}
Though the above result is stated for TASEP, the hypothesis is for the associated last passage percolation model. Let us unfold what this means in terms of the TASEP height function initial data. The height function initial data is given by the function $s\mapsto h_N(s;0)$. The first part of the hypothesis regards the limiting behavior of the central part of this function. In other words, the function $s\mapsto N^{-1/3} h_N(N^{2/3} s;0)$ should have a limit (related to the function $\ell(s)$ which is assumed not to grow too quickly with $|s|$). Outside of a large window around the origin, this scaled height function initial data should not go to $-\infty$ too quickly (or else it may influence the later time behavior around the origin). The second part of the hypothesis describes a growth condition which ensures this.
\end{rmk}

Below we list several typical initial conditions of TASEP, and their initial height functions. We characterize the initial height function $h(s; 0)$ only at integer-valued $s$.  Note that all the initial conditions are $N$-independent. Theorem \ref{thm:main_TASEP} covers more general, $N$-dependent initial conditions, for example, periodic initial conditions with period $\bigO(N^{2/3})$.
\begin{itemize}
\item (Step initial condition)
  Initially all negative sites are occupied and all non-negative sites are empty, \ie,
  \begin{equation}
    h^{\step}(s; 0) = \lvert s \rvert.
  \end{equation}
\item (Flat initial condition)
  Initially all even sites are occupied and all odd sites are empty, \ie,
  \begin{equation} \label{eq:initial_flat}
    h^{\flatinit}(s; 0) =
    \begin{cases}
      0 & \text{if $s = 0, \pm 2, \pm 4, \dotsc$}, \\
      -1 & \text{if $s = \pm 1, \pm 3, \dotsc$}, \\
    \end{cases}
  \end{equation}

\item (Brownian/Bernoulli initial condition)
  Initially all sites are independently occupied with probability $\frac{1}{2}$ and empty with probability $1/2$, \ie,
  \begin{equation} \label{eq:initial_Bernoulli}
    h^{\Bernoulli}(s; 0) =
    \begin{cases}
      0 & \text{if $s = 0$}. \\
      \sum^{s - 1}_{i = 0} w_i & \text{if $s = 1, 2, \dotsc$}, \\
      -\sum^{-1}_{i = -s} w_i & \text{if $s = -1, -2, \dotsc$}, \\
    \end{cases}
  \end{equation}
  and $w_i$ are \iid\ Bernoulli random variables such that $\Prob(w_i = 1) = 1/2$ and $\Prob(w_i = -1) = 1/2$.

\item (Wedge-flat initial condition)
  Initially all negative sites and all even sites are occupied, but all positive odd sites are empty, \ie,
  \begin{equation} \label{eq:initial_Wedge_flat}
    h^{\wedgeflat}(s; 0) =
    \begin{cases}
      h^{\flatinit}(s; 0) & \text{if $s \geq 0$}, \\
      h^{\step}(s; 0) & \text{if $s < 0$}.
    \end{cases}
  \end{equation}

\item (Wedge-Bernoulli initial condition)
  Initially all negative sites are occupied, and all non-negative sites are independently occupied with probability $1/2$ and empty with probability $1/2$, \ie,
  \begin{equation} \label{eq:initial_Wedge_Bernoulli}
    h^{\wedgeBernoulli}(s; 0) =
    \begin{cases}
      h^{\Bernoulli}(s; 0) & \text{if $s \geq 0$}, \\
      h^{\step}(s; 0) & \text{if $s < 0$}.
    \end{cases}
  \end{equation}

\item (Flat-Bernoulli initial condition)
  Initially all even negative sites are occupied, all odd negative sites are empty, and all non-negative sites are independently occupied with  probability $1/2$ and empty with probability $1/2$, \ie,
  \begin{equation} \label{eq:initial_flat_Bernoulli}
    h^{\flatBernoulli}(s; 0) =
    \begin{cases}
      h^{\Bernoulli}(s; 0) & \text{if $s \geq 0$}, \\
      h^{\flatinit}(s; 0) & \text{if $s < 0$}.
    \end{cases}
  \end{equation}
\end{itemize}

As consequences of Theorem \ref{thm:main_TASEP} we can prove variational formulas for one-point distributions of TASEP started from initial data as in \eqref{eq:initial_flat}, \eqref{eq:initial_Bernoulli} \eqref{eq:initial_Wedge_flat}, \eqref{eq:initial_Wedge_Bernoulli} and \eqref{eq:initial_flat_Bernoulli}. To state the results in a uniform way, we denote the two-sided Brownian motion $\B(s)$ by
  \begin{equation} \label{eq:defn_two_sided_BM}
    \B(s) =
    \begin{cases}
      \B_+(s) & \text{if $s \geq 0$}, \\
      \B_-(-s) & \text{if $s \leq 0$},
    \end{cases}
  \end{equation}
  where $\B_+(s)$ and $\B_-(s)$ are independent standard Brownian motions starting at $0$.
\begin{cor} \label{thm:Bernoulli_initial_condition}
  Let $\sigma$ be a real constant and $h(s; t)$ be the height function of the TASEP.
  \begin{enumerate}[label=(\alph*)]
  \item
    With the flat initial condition \eqref{eq:initial_flat},
    \begin{equation}
      \lim_{N \to \infty} \Prob \left( \frac{h^{\flatinit}(2\sigma \cnought N^{2/3}; \anought^* N) - 2N}{2\dnought^* N^{1/3}} > -x \right) = \Prob \left( \max_{s \in \realR} \left( \AiryA(s) - s^2 \right) < x \right).
    \end{equation}
  \item \label{enu:thm:Bernoulli_initial_condition:a}
    With the Bernoulli initial condition \eqref{eq:initial_Bernoulli}, 
    \begin{equation} \label{eq:Bernoulli_init_result}
      \lim_{N \to \infty} \Prob \left( \frac{h^{\Bernoulli}(2\sigma \cnought N^{2/3}; \anought^* N) - 2N}{2\dnought^* N^{1/3}} > -x \right) = \Prob \left( \max_{s \in \realR} \left( \AiryA(s) - (s - \sigma)^2 + \sqrt{2} q^{-1/4} \B(s) \right) < x \right).
    \end{equation}
  \item \label{enu:thm:Bernoulli_initial_condition:b}
    With the Wedge-flat initial condition \eqref{eq:initial_Wedge_flat},
    \begin{equation} \label{eq:wedge_flat_init_result}
      \lim_{N \to \infty} \Prob \left( \frac{h^{\wedgeflat}(2\sigma \cnought N^{2/3}; \anought^* N) - 2N}{2\dnought^* N^{1/3}} > -x \right) = \Prob \left( \max_{s \le \sigma} \left( \AiryA(s) - s^2 \right) < x \right).
    \end{equation}
  \item \label{enu:thm:Bernoulli_initial_condition:c}
    With the Wedge-Bernoulli initial condition \eqref{eq:initial_Wedge_Bernoulli},
    \begin{multline} \label{eq:wedge_Bernoulli_init_result}
      \lim_{N \to \infty} \Prob \left( \frac{h^{\wedgeBernoulli}(2\sigma \cnought N^{2/3}; \anought^* N) - 2N}{2\dnought^* N^{1/3}} > -x \right) \\
      = \Prob \left( \max_{s \ge 0} \left( \AiryA(s) - (s-\sigma)^2 +\sqrt{2} q^{-1/4} \B(s) \right) < x \right).
    \end{multline}
  \item \label{enu:thm:Bernoulli_initial_condition:d}
    With the Flat-Bernoulli condition \eqref{eq:initial_flat_Bernoulli},
    \begin{multline} \label{eq:flat_Bernoulli_init_result}
      \lim_{N \to \infty} \Prob \left( \frac{h^{\flatBernoulli}(2\sigma \cnought N^{2/3}; \anought^* N) - 2N}{2\dnought^* N^{1/3}} > -x \right) = \\
      \Prob \left( \max_{s \in \realR} \left( \AiryA(s) - (s-\sigma)^2 +\sqrt{2} q^{-1/4} \chi_{s\ge 0} \B(s) \right) < x \right).
    \end{multline}
  \end{enumerate}
\end{cor}
\begin{rmk}
  The result for the flat initial condition \eqref{eq:initial_flat} is obtained in \cite{Johansson03} and is given, in an equivalent form, in Proposition \ref{prop:Johansson_max_Ai} in the case that $\sigma = 0$. Since the flat initial condition is translational invariant, the result holds for general $\sigma$. The step initial condition is singular in the sense that $K_1 = K_2 = 0$ in \eqref{eq:L_by_h} and hence $a_N = b_N = 0$ and the interval $(a_{\infty}, b_{\infty})$ degenerates a point $\{ 0 \}$. The result, which is stated in Proposition \ref{prop:previous_pt_to_pt}\ref{enu:prop:previous_pt_to_pt_3}, actually is used in the proof of Theorem \ref{thm:main_TASEP} (and Theorem \ref{thm:main}), so we do not list it as a corollary.
\end{rmk}

\begin{rmk}
For discrete time parallel update TASEP, Bernoulli initial data is {\it not} stationary in time. The stationary measure is not a product measure. This explains why besides in the limiting case when $q=1$, there is no known (simple) formula in the literature of the right-hand side of (\ref{eq:Bernoulli_init_result}).
\end{rmk}

Comparing the results in Corollary \ref{thm:Bernoulli_initial_condition} with the asymptotics of $h(s; t)$ obtained in continuous TASEP models (see \cite{Baik-Ferrari-Peche10}, \cite{Borodin-Ferrari-Sasamoto08a}, \cite{Baik-Ben_Arous-Peche05}, \cite{Borodin-Ferrari-Sasamoto09} for details) that corresponds to the $q \to 1_-$ limit of the discrete TASEP model considered in this paper, one expects the following results:
\begin{align}
  \Prob \left( \max_{t \in \realR} \left( \AiryA(s) - (s - \sigma)^2 \right) < x \right) = {}& \Prob(2^{1/3} \AiryA_1(2^{-2/3} \sigma) < x), \label{eq:flat_limit} \\
  \Prob \left( \max_{t \in \realR} \left( \AiryA(s) - (s - \sigma)^2 + \sqrt{2}\B(s) \right) < x \right) = {}& \Prob(\AiryA_{\stat}(\sigma) < x), \label{eq:Bernoulli_limit} \\
  \Prob \left( \max_{s \ge 0} \left( \AiryA(s) - (s -  \sigma)^2 \right) < x \right) = {}& \Prob(\AiryA_{2 \to 1}(\sigma) < x + \sigma^2 \chi_{\sigma < 0}), \label{eq:step_flat_limit} \\
  \Prob \left( \max_{s \ge 0} \left( \AiryA(s) - (s-\sigma)^2 +\sqrt{2} \B(s) \right) < x \right) = {}& \Prob(\AiryA_{\BM \to 2}(-\sigma) < x + \sigma^2), \label{eq:step_Bernoulli_limit} \\
  \Prob \left( \max_{s \in \realR} \left( \AiryA(s) - (s-\sigma)^2 +\sqrt{2} \chi_{s\ge 0} \B(s) \right) < x \right) = {}& \Prob(\AiryA_{2 \to 1, 1, 0}(-\sigma) < x + \sigma^2 \chi_{\sigma > 0}). \label{eq:flat_Bernoulli_limit}
\end{align}
(It may be possible to prove these continuous analogs in a similar manner as done here, but we do not pursue this presently, see Remark \ref{rmk:continuous_limit}.)

Below are explanations of notations:
\begin{itemize}
\item
  In \eqref{eq:flat_limit}, $\AiryA_1$ stands for the Airy process with flat initial data, defined in \cite{Sasamoto05} and \cite[Formulas (1.4) and (1.5)]{Borodin-Ferrari-Prahofer-Sasamoto07}. The $\AiryA_1$ process is stationary, and its $1$-dimensional distribution is \cite{Ferrari-Spohn05}
  \begin{equation}
    \Prob(\AiryA_1(\sigma) < x) = F_{\GOE}(2x),
  \end{equation}
  where $F_{\GOE}$ is the GOE Tracy-Widom  distribution \cite{Tracy-Widom96}.
\item
  In \eqref{eq:Bernoulli_limit}, $\AiryA_{\stat}$ stands for the Airy process with stationary initial data, defined in \cite{Baik-Ferrari-Peche10}, and we follow the notation in \cite[Section 1.11]{Quastel12} and \cite[Section 1.2]{Quastel-Remenik13}. The $1$-dimensional distribution of $\AiryA_{\stat}(\sigma)$ appears also in literature as (see \cite[Remark 1.3]{Baik-Ferrari-Peche10} and \cite[Appendix A]{Ferrari-Spohn06})
\begin{equation}
  \Prob(\AiryA_{\stat}(\sigma) < x) = F_{\sigma}(x) = H \left( x + \sigma^2; \frac{\sigma}{2}, -\frac{\sigma}{2} \right),
\end{equation}
where $F_{\sigma}(x)$ is defined in \cite[Formula (1.20)]{Ferrari-Spohn06} and $H(x;w_+,w_-)$ is defined in \cite[Definition 3]{Baik-Rains00}.
\item
  In \eqref{eq:step_flat_limit}, the transition process $\AiryA_{2 \to 1}$ interpolating the $\AiryA_2$ and $\AiryA_1$ processes is introduced in \cite[Definition 2.1]{Borodin-Ferrari-Sasamoto08a} (see also \cite[Formula (1.7)]{Quastel-Remenik13a}, where the notation for the right-hand side of \eqref{eq:step_flat_limit} is $G_{\sigma}^{2\to 1}(x+\sigma^2\chi_{\sigma<0})$).
\item
  In \eqref{eq:step_Bernoulli_limit}, the transition process $\AiryA_{\BM \to 2}$ interpolating the Brownian motion and $\AiryA_2$ process is introduced in \cite[Formula (3.6)]{Imamura-Sasamoto04}, see also \cite[Definition 2.13]{Corwin-Ferrari-Peche10}. The $1$-dimensional distribution of $\AiryA_{\BM \to 2}(\sigma)$ was conjectured in \cite{Prahofer-Spohn02a} and proved in \cite{Ben_Arous-Corwin11} to be
\begin{equation}
  \Prob(\AiryA_{\BM \to 2}(\sigma) < x) = F_1(x; \sigma),
\end{equation}
where the distribution function $F_1$ is introduced in \cite[Definition 1.3]{Baik-Ben_Arous-Peche05}.
\item
  In formula \eqref{eq:flat_Bernoulli_limit}, The transition process $\AiryA_{2 \to 1, 1, 0}$ interpolating the Brownian motion and the $\AiryA_1$ process is introduced in \cite[Definition 18]{Borodin-Ferrari-Sasamoto09}. It is defined from the TASEP with one slow particle, and it is related to the TASEP with flat-Bernoulli initial condition via Burke's theorem, as explained in \cite{Borodin-Ferrari-Sasamoto09}.
\end{itemize}

Among formulas \eqref{eq:flat_limit}, \eqref{eq:Bernoulli_limit}, \eqref{eq:step_flat_limit}, \eqref{eq:step_Bernoulli_limit} and \eqref{eq:flat_Bernoulli_limit}, \eqref{eq:flat_limit} is proved in \cite{Johansson03}, and then proved in a direct way in \cite{Corwin-Quastel-Remenik11}. Formula \eqref{eq:step_flat_limit} is proved in \cite{Quastel-Remenik13a}. Formulas \eqref{eq:Bernoulli_limit}, \eqref{eq:step_Bernoulli_limit} and \eqref{eq:flat_Bernoulli_limit} are conjectured in \cite[Section 1.4]{Quastel-Remenik13}. Note that in \cite[Section 1.4]{Quastel-Remenik13}, the notations $\AiryA_{1 \to BM}$ and $\AiryA_{2 \to BM}$ are described but not precisely defined. From the context we figure out that
\begin{equation}
  \AiryA_{2 \to \BM}(\sigma) = \AiryA_{\BM \to 2}(-\sigma) - \sigma^2 \chi_{\sigma > 0}, \quad \AiryA_{1 \to \BM}(\sigma) = \AiryA_{2 \to 1, 1, 0}(-\sigma) - \sigma^2 \chi_{\sigma > 0}.
\end{equation}
Formulas \eqref{eq:Bernoulli_limit} and \eqref{eq:step_Bernoulli_limit} are special cases of Corollary \ref{cor:special_cases_inhomogeneous} (c) with $w_+=-w_-=\frac{\sigma}{2}$ and (a) with $k=1$  in Section \ref{Sec.LPP_inhomogeneous}. And our argument in this paper is a strong support to the conjectural formula \eqref{eq:flat_Bernoulli_limit}.

\begin{rmk} \label{rmk:continuous_limit}
  The method in our study of the discrete time TASEP, if applied on the continuous time TASEP, that is, the $q \to 1_-$ limit of the discrete time one, yields the counterparts of \eqref{eq:Bernoulli_init_result}, \eqref{eq:wedge_flat_init_result}, \eqref{eq:wedge_Bernoulli_init_result} and \eqref{eq:flat_Bernoulli_init_result} with $q = 1$, and then the formulas \eqref{eq:flat_limit}, \eqref{eq:Bernoulli_limit}, \eqref{eq:step_flat_limit}, \eqref{eq:step_Bernoulli_limit} and \eqref{eq:flat_Bernoulli_limit} are derived directly. The only technical obstacle in the application of our method in the continuous time TASEP is that the counterpart of Proposition \ref{prop:Johansson_max_Ai}, where the discrete geometric distribution of $w(i, j)$ is replaced by the continuous exponential distribution is not available in literature. We remark that the counterpart of Proposition \ref{prop:Johansson_max_Ai} can also be proved by the method in \cite{Johansson03}.
\end{rmk}

\subsection{LPP with inhomogeneous parameter geometric weight distributions}
\label{Sec.LPP_inhomogeneous}
In this subsection we consider the point-to-point LPP on a $\intZ^2$ lattice where the weights on sites are in independent geometric distribution, but with nonidentical parameters. The strategy is to express the point-to-point LPP with respect to these weights by point-to-curve LPP with respect to homogeneous weight as considered in Section \ref{subsec:LPP_on_Z*Z}.

Let $L$ be the vertical path (depending on $N$ which we suppress)
\begin{equation}
  L := \{ (0, y) \mid y \in D_N \}, \quad \text{where $D_N$ is an interval on $\realR$.}
\end{equation}
We are most interested in the case that $D_N = \realR$. But the LPP $G_{(N, N)}(L)$ is not well defined in this case, since $G_{(N, N)}(0, y) \to +\infty$ almost surely as $y \to -\infty$. We consider a modified LPP  
\begin{equation}
  G^{f_N}_{(N, N)}(L) = \max_{y \in D_N}\,\left( G_{(N, N)}(0, y) - f_N(y)\right)
\end{equation}
where $f_N: D_N \to \realR$ is a function where $D_N$, the domain of $f_N$, is an interval. This modified LPP $G^{f_N}_{(N, N)}(L)$ is well defined for $D_N = \realR$ if $f_N(x) \to +\infty$ fast enough as $x \to -\infty$.

By Proposition \ref{prop:previous_pt_to_pt}, for $y = cN$ where $c$ is in a compact subset of $(-\infty, 1)$, if $f_N(y) = a_0(1 - y/N)N$, then $G_{(N, N)}(0, y) = o(N)$ with high probability. So if $f_N(y)$ is close to $\anought N = a_0(1)N$ for $y$ around $0$, and otherwise greater than $a_0(1 - y/N)N$ for all $y < N$, then $G^{f_N}_{(N, N)}(L)$ is $o(N)$ and the value of $y$ such that $G_{(N, N)}(0, y) - f_N(y)$ attains its maximum is in the vicinity of $0$ with high probability. To make the idea above precise, we state a technical hypothesis for $f_N$ analogous to the hypotheses of Definition \ref{thmhypo}.

Before stating the technical hypotheses, we consider another similar question regarding putting a function on the $\llcorner$-shaped path
\begin{equation} \label{eq:L_shaped_L_tilde}
  \tilde{L} = \big\{ (0, y) \mid y \geq 0 \} \cup \{ (x, 0) \mid x \geq 0 \big\}.
\end{equation}
The random variable $G_{(N, N)}(\tilde{L})$ is well defined and equivalent to the point-to-point last passage time $G_{(N, N)}(0, 0)$. If, for $D_N=[-N,N]$, $\tilde{f}_N:D_N \to \realR$ is a continuous function such that $\tilde{f}_N(x)$ increases suitably fast as $\lvert x \rvert$ increases, then the modified LPP
\begin{equation}
  G^{\tilde{f}_N}_{(N, N)}(\tilde{L}) = \max \left( \max_{y \geq 0} \left(G_{(N, N)}(0, y) - \tilde{f}_N(y)\right),\ \max_{x \geq 0} \left(G_{(N, N)}(x, 0) - \tilde{f}_N(-x)\right) \right)
\end{equation}
has a nontrivial limit like we observed for $G^{f_N}_{(N, N)}(L)$. The function $\tilde{f}_N$ is serving as a boundary condition for the last passage problem on the positive coordinate axes. We turn now to the hypotheses needed to state a fluctuation theorem regarding these modified LPP problems.

\begin{definition} \label{inhomogeneous_hypo}
Consider constants
\begin{equation}
C>0,\quad c_1\in (0,1),\quad c_2\in (0,1/3),\quad c_3>0,\quad a_{\infty} \in \{ -\infty \} \cup \realR, \quad b_{\infty} \in \{ +\infty \} \cup \realR,
\end{equation}
and a sequence $\{m_N\}_{N\geq 1}\subset \realR_{+}$ converging to zero as $N$ goes to infinity.

We say that a continuous function $\ell:\realR\to \realR$ and a sequence of functions $f_N:D_N\to \realR$ satisfy $\mathrm{Hyp}^{\vert}\big(C,c_1,c_2,c_3,a_{\infty},b_{\infty},\{m_N\}\big)$ if the following properties hold:
\begin{enumerate}
\item The function $\ell(s)$ satisfies the bound for $s\in \realR$ that
\begin{equation}
  \ell(s) < C + c_1 s^2.
\end{equation}
\item There is a sequence of intervals $I_N = (a_N, b_N) \subset (-N^{c_2},N^{c_2})$ converging to $(a_{\infty},b_{\infty})$ such that
     \begin{equation} \label{eq:f_N_char}
      f_N(2s \cnought N^{2/3}) = \anought N - s\anought\cnought N^{2/3} - (\ell(s) + l_N(s)) \dnought N^{1/3},
    \end{equation}
    where $l_N(s): I_N \to \realR$ is some continuous function with $\max_{s \in I_N}\  \lvert l_N(s) \rvert \leq m_N$ and $\lim_{N \to \infty} m_N = 0$,
    and $\anought, \dnought$ are defined in \eqref{eq:defn_nought}.
\item  For all $y \in D_N$ such that $yN^{1/3}/(2\cnought)\in (-\infty, N^{1/3}/(2\cnought)]\setminus I_N$, $f_N(yN)$ satisfies the inequality 
    \begin{equation} \label{eq:outer_f_N^max}
      \frac{f_N(yN)}{N} >
        \max \left( \anought - \frac{\anought y}{2} - c_1 \dnought \left( \frac{y}{2\cnought} \right)^2,\ a_0(1 - y) + c_3 \lvert y \rvert \right).  
    \end{equation}
\end{enumerate}

Let us also define a second hypothesis. Assume $C,c_1,c_2,c_3,a_{\infty},b_{\infty},\{m_N\}$ are as above and that $D_N=[-N,N]$. We say that $\ell$ and the sequence $\tilde{f}_N:D_N\to \realR$ satisfy $\mathrm{Hyp}^{\llcorner}\big(C,c_1,c_2,c_3,a_{\infty},b_{\infty},\{m_N\}\big)$ if the above conditions hold with (\ref{eq:f_N_char}) replaced by
\begin{equation}
      \tilde{f}_N(2s \cnought N^{2/3}) = \anought N - \lvert s\rvert\anought\cnought N^{2/3} - (\ell(s) + l_N(s)) \dnought N^{1/3},
\end{equation}
and (\ref{eq:outer_f_N^max}) replaced by
\begin{equation}
      \frac{\tilde{f}_N(yN)}{N} > \max \left( \anought N - \frac{\anought \lvert y \rvert}{2} - c_1 \dnought \left( \frac{y}{2\cnought} \right)^2,\ a_0(1 - \lvert y \rvert) + c_3 \lvert y \rvert \right).
\end{equation}
\end{definition}

\begin{thm} \label{thm:inhomogeneous}
Fix constants $C, c_1, c_2, c_3, a_{\infty}, b_{\infty}$, a sequence $\{m_N\}_{N\geq 1}$ as in Definition \ref{inhomogeneous_hypo}. Then for all $\epsilon>0$ there exists $N_0$ such that for all continuous functions $\ell:\realR\to \realR$ and sequences of functions $f_N:D_N\to \realR$ satisfying $\mathrm{Hyp}^{\vert}\big(C,c_1,c_2,c_3,a_{\infty},b_{\infty},\{m_N\}\big)$, and for all $N>N_0$ and $x\in \realR$,
 \begin{equation}
    \left\lvert \Prob \left( \frac{G^{f_N}_{(N, N)}(L)}{\bnought N^{\frac{1}{3}}} < x \right) - \Prob \left( \max_{s \in (a_\infty,b_\infty)} \left( \AiryA(s) - s^2 + \ell(s) \right) < x \right) \right\rvert < \epsilon.
  \end{equation}
With $D_N=[-N,N]$, the sequence $f_N$ replaced by $\tilde{f}_N$, the hypothesis $\mathrm{Hyp}^{\vert}\big(C,c_1,c_2,c_3,a_{\infty},b_{\infty},\{m_N\}\big)$ replaced by $\mathrm{Hyp}^{\llcorner}\big(C,c_1,c_2,c_3,a_{\infty},b_{\infty},\{m_N\}\big)$, and the LPP problem $G^{f_N}_{(N, N)}(L)$ replaced by $G^{\tilde{f}_N}_{(N, N)}(\tilde{L})$ the above limiting result likewise holds true.
\end{thm}

As applications of Theorem \ref{thm:inhomogeneous} (or adaption of its, see Remark \ref{rmk:not_real_corollaries}), we have the following results for point-to-point LPP with inhomogeneous parameter geometric weight distributions. The weight parameters we will consider differ from the homogeneous ones considered in Section \ref{subsec:LPP_on_Z*Z} in only finitely many columns and/or rows. 
So we use the same notation $\G(N, N)$ which is defined in \eqref{eq:defn_inverse} and \eqref{eq:LPP_defn}, but the weights on some of the lattice points are defined differently. To state the following corollaries, we denote by $\AiryA^{(1)}$ and $\AiryA^{(2)}$  two  independent Airy processes that are the $\AiryA$ described in Section \ref{Sec.distributions}, and denote by $\B_1, \dotsc, \B_k$ independent two-sided Brownian motions that are the $\B$ defined in \eqref{eq:defn_two_sided_BM}.
\begin{cor} \label{cor:corrollaries_of_inhomogeneous_LPP}
  In the $\intZ^2$ lattice we consider the point-to-point LPP $\G(N, N)$, and denote
  \begin{equation} \label{eq:defn_tilde_G}
    \tilde{G}_N = \frac{\G(N, N) - \anought N}{\bnought N^{1/3}},
  \end{equation}
  where $\anought$ and $\bnought$ are defined in \eqref{eq:defn_nought}.
  \begin{enumerate}[label=(\alph*)]
  \item \label{thm:spiked_border}
  Suppose the weights $w(i, j)$ are independent and geometrically distributed with parameter $\alpha_{i, j}$ such that $\alpha_{i, j} = 1 - q$ if $i \notin \{ 0, 1, \dotsc, k - 1 \}$ and
  \begin{equation} \label{eq:weight_k_borders}
    \alpha_{i - 1, j} = 1 - \sqrt{q} \left( 1 - \frac{2w_i}{\dnought N^{1/3}} \right) \quad \text{if} \quad i = 1, \dotsc, k,
  \end{equation}
  where $k \in \intZ_+$ and $w_1, \dotsc, w_k \in \realR$ are constants. Then
  \begin{multline} \label{eq:spiked_border}
    \lim_{N \to \infty} \Prob (\tilde{G}_N \leq x) = \\
    \Prob \left( \max_{0 = s_0 \leq s_1 \leq \dotsb \leq s_k} \left( \AiryA(s_k) + \sqrt{2} \sum^k_{i = 1} (\B_i(s_i) - \B_i(s_{i - 1})) - 4\sum^k_{i = 1} w_i(s_i - s_{i - 1}) - s^2_k \right) \leq x \right).
  \end{multline}
\item \label{thm:double_spiked_border}
  Suppose the weight $w(0, 0)$ is fixed to be $0$, the weights $w(i, j)$ are independent and geometrically distributed with parameter $\alpha_{i, j}$ if $i, j$ are not both $0$, such that $\alpha_{i, j} = 1 - q$ if $i, j$ are both nonzero, and
  \begin{equation}
    \alpha_{i, j} =
    \begin{cases}
      1 - \sqrt{q} \left( 1 - \frac{2w_+}{\dnought N^{1/3}} \right) & \text{if $i \geq 1$ and $j = 0$}, \\
      1 - \sqrt{q} \left( 1 - \frac{2w_-}{\dnought N^{1/3}} \right) & \text{if $i = 0$ and $j \geq 1$}.
    \end{cases}
  \end{equation}
  where $w_+, w_- \in \realR$ are constants. Then
  \begin{equation} \label{eq:double_spiked_border}
    \lim_{N \to \infty} \Prob (\tilde{G}_N \leq x) =  \Prob \left( \max_{s \in \realR} \left( \AiryA(s) + \sqrt{2}\B(s) + 4(w_+ 1_{s < 0} - w_- 1_{s > 0})s - s^2 \right) \leq x \right).
  \end{equation}
\item \label{thm:inner_spiked_border}
  Suppose the weight $w(i, j)$ are independent and geometrically distributed with parameter $\alpha_{i, j}$ such that $\alpha_{i, j} = 1 - q$ if $j \leq [\alpha N]$ or $j > [\alpha N] + k$, and
  \begin{equation}
    \alpha_{i, j} = 1 - \sqrt{q} \left( 1 - \frac{2w_{j - [\alpha N]}}{\dnought N^{1/3}} \right) \quad \text{if} \quad j = [\alpha N] + 1, \dotsc, [\alpha N] + k,
  \end{equation}
  where $\alpha \in (0, 1)$, $k \in \intZ_+$ and $w_1, \dotsc, w_k \in \realR$ are constants. Then
  \begin{multline} \label{eq:inner_spiked_border}
    \lim_{N \to \infty} \Prob (\tilde{G}_N \leq x) = \Prob \left( \max_{s_0 \leq s_1 \leq \dotsb \leq s_k} \left( \alpha^{1/3} \AiryA^{(1)}(\alpha^{-2/3} s_0) + \sqrt{2} \sum^k_{i = 1} (\B_i(s_i) - \B_i(s_{i - 1})) \right. \right. \\
    + \left. \left. ( 1 - \alpha)^{1/3} \AiryA^{(2)}((1 - \beta)^{-2/3} s_k)
        \vphantom{\max_{s_0 \leq s_1 \leq \dotsb \leq s_k}\sum^k_{i = 1}} - 4\sum^k_{i = 1} w_i(s_i - s_{i - 1}) - \frac{s^2_0}{\alpha} - \frac{s^2_k}{1 - \alpha} \right) \leq x \right).
  \end{multline}
  \end{enumerate}
\end{cor}
\begin{rmk} \label{rmk:not_real_corollaries}
  Parts \ref{thm:spiked_border} and \ref{thm:double_spiked_border} of Corollary \ref{cor:corrollaries_of_inhomogeneous_LPP} are direct consequences of the first and second parts of Theorem \ref{thm:inhomogeneous}. Part \ref{thm:inner_spiked_border} does not follow this theorem in a straightforward way, although the proof of the theorem can be adapted to prove Part \ref{thm:inner_spiked_border}.
\end{rmk}

The limits on the left-hand sides of \eqref{eq:spiked_border}, \eqref{eq:double_spiked_border} and \eqref{eq:inner_spiked_border} have been analyzed previously in \cite{Baik-Rains00}, \cite{Baik-Ben_Arous-Peche05} and \cite{Baik06}, and the results were given in other forms by Fredholm determinants. Utilizing these earlier results we arrive at the following expressions for these statistics.
\begin{cor} \label{cor:special_cases_inhomogeneous}
  For all $x \in \realR$,
  \begin{enumerate}[label=(\alph*)]
  \item
    for all parameters $w_1, \dotsc, w_k \in \realR$,
    \begin{multline}
      \Prob \left( \max_{0 = s_0 \leq s_1 \leq \dotsb \leq s_k} \left(\AiryA(s_k) + \sqrt{2} \sum^k_{i = 1} (\B_i(s_i) - \B_i(s_{i - 1})) - 4\sum^k_{i = 1} w_i(s_i - s_{i - 1}) - s^2_k\right) \leq x \right) = \\
      F^{\spiked}_k(x; 2w_1, \dotsc, 2w_k),
    \end{multline}
    \item
    for all parameters $\alpha \in (0, 1)$ and $w_1, \dotsc, w_k \in \realR$,
    \begin{multline}
      \Prob \left( \max_{s_0 \leq s_1 \leq \dotsb \leq s_k} \left( \alpha^{1/3} \AiryA^{(1)}(\alpha^{-2/3} s_0) + \sqrt{2} \sum^k_{i = 1} \left(\B_i(s_i) - \B_i(s_{i - 1})\right) + ( 1 - \alpha)^{1/3} \A^{(2)}((1 - \beta)^{-2/3} s_k) \right. \right. \\
      \left. \left. \vphantom{\max_{s_0 \leq s_1 \leq \dotsb \leq s_k}\sum^k_{i = 1}} - 4\sum^k_{i = 1} w_i(s_i - s_{i - 1}) - \frac{s^2_0}{\alpha} - \frac{s_k^2}{1 - \alpha} \right) \leq x \right) = F^{\spiked}_k(x; 2w_1, \dotsc, 2w_k), \\
    \end{multline}
    \item
    for all parameters $w_+, w_- \in \realR$,
    \begin{equation}
      \Prob \left( \max_{s \in \realR} \left( \AiryA(s) + \sqrt{2}\B(s) + 4(w_+ 1_{s < 0} - w_- 1_{s > 0})s - s^2 \right) \leq x \right) = H(x; w_+, w_-),
    \end{equation}
  \end{enumerate}
  where $F^{\spiked}_k(x; w_1, \dotsc, w_n)$ is the distribution introduced in \cite[Formula (54)]{Baik-Ben_Arous-Peche05} and \cite[Corollary 1.3]{Baik06}, and $H(x; w_+, w_-)$ is the distribution function introduced in \cite{Baik-Rains00}.
\end{cor}
\subsection{The Airy process}\label{Sec.distributions}

The Airy process $\AiryA(\cdot)$ \cite{Prahofer-Spohn02} (sometimes also denoted as $\AiryA_2(\cdot)$ and called the Airy$_2$ process, in contrast to the Airy$_1$ process $\AiryA_1$ considered in \eqref{eq:flat_limit}) is an important process appearing in the Kardar-Parisi-Zhang universality class, see for example \cite{Corwin11}. Its properties have been intensively studied, see for example \cite{Johansson03}, \cite{Corwin-Hammond11}, \cite{Quastel-Remenik13}.

The Airy process $\AiryA(\cdot)$ is defined through its finite-dimensional distributions which are given by a Fredholm determinant formula. For $x_0,\ldots,x_n\in\mathbb{R}$ and $t_0<\ldots<t_n$ in $\mathbb{R}$,
\begin{equation}\label{eq:detform}
\mathbb{P}\big(\AiryA(t_0)\le x_0,\ldots,\AiryA(t_n)\le x_n\big ) =
\det(I-\mathrm{f}^{1/2}K_{\mathrm{ext}}\mathrm{f}^{1/2})_{L^2(\{t_0,\ldots,t_n\}\times\mathbb{R})},
\end{equation}
where we have counting measure on $\{t_0,\ldots,t_n\}$ and Lebesgue measure on $\mathbb{R}$, $\mathrm f$ is defined on
$\{t_0,\ldots,t_n\}\times\mathbb{R}$ by $\mathrm{f}(t_j,x)=\mathbf{1}_{x\in(x_j,\infty)}$,
and the {\it extended Airy kernel} \cite{Prahofer-Spohn02} is defined by
$$K_\mathrm{ext}(t,\xi;t',\xi')=
\begin{cases}
\int_0^\infty d\lambda\,e^{-\lambda(t-t')}\Ai(\xi+\lambda)\Ai(\xi'+\lambda), &\text{if $t\ge t'$}\\
-\int_{-\infty}^0 d\lambda\,e^{-\lambda(t-t')}\Ai(\xi+\lambda)\Ai(\xi'+\lambda),  &\text{if $t<t'$},
\end{cases}$$
where $\Ai(\cdot)$ is the Airy function. It is readily seen that the Airy process is stationary. The one point distribution of $\AiryA$ is the $F_{\GUE}$ distribution (i.e., the GUE Tracy-Widom  distribution \cite{Tracy-Widom96}).

Since our main results appear as variational problems involving the Airy process, it is important to know that these problems are well-posed with finite answers. It was proved in \cite[Theorem 4.3]{Prahofer-Spohn02} and \cite[Theorem 1.2]{Johansson03} that there exists a measure on $\mathcal{C}(\realR,\realR)$ (continuous functions from $\realR\to \realR$ endowed with the topology of uniform convergence on compact subsets) whose finite dimensional distributions coincide with those of the Airy process (i.e., there exists a continuous version of the Airy process). Further properties of the Airy process were demonstrated in \cite{Corwin-Hammond11}. We summarize those properties which we will appeal to. Part \ref{enu:prop:Airy_property:a} of Proposition \ref{prop:Airy_property} is a special case of \cite[Proposition 4.1]{Corwin-Hammond11}, (our $\AiryA(t)$ is their $\mathcal{A}_1(t)$), while Part \ref{enu:prop:Airy_property:b} is a generalization of \cite[Proposition 4.4]{Corwin-Hammond11} where the parameter $c$ is taken as $1$, and the proof can be used for our generalized case with little modification.
\begin{prop} \label{prop:Airy_property}
  \begin{enumerate}[label=(\alph*)]
  \item (Local Brownian absolute continuity) \label{enu:prop:Airy_property:a}
    For any $s, t \in \realR$, $t > 0$, the measure on functions from $[0, t] \to \realR$ given by $\AiryA( \cdot + s) - \AiryA(s)$ is absolutely continuous with respect to Brownian motion of diffusion parameter $2$.
  \item  \label{enu:prop:Airy_property:b}
    For all positive constants $\alpha$ and $c$ such that $\alpha < c$, there exists $\epsilon > 0$ and $C(\alpha, c) > 0$ such that for all $t \geq C(\alpha, c) > 0$ and $x \geq -\alpha t^2$,
    \begin{equation}
      \Prob(\sup_{s \notin [-t, t]}  (\AiryA(s) - cs^2) > x) \leq e^{-\epsilon(ct^2 + x)^{3/2}}.
    \end{equation}
  \end{enumerate}
\end{prop}

One direct consequence of Proposition \ref{prop:Airy_property} is the well-definedness of the limit distributions in Theorems \ref{thm:main}, \ref{thm:main_TASEP}, and \ref{thm:inhomogeneous}.
\begin{cor} \label{cor:well_defined_ness}
  Let $\ell: \realR \to \realR$ be a continuous function that satisfies \eqref{eq:upper_bound_l} and $(a_{\infty}, b_{\infty})$ be an interval such that $-\infty \leq a_{\infty} < b_{\infty} \leq +\infty$. Then $\max_{s \in (a_{\infty}, b_{\infty})} \left(\AiryA(s) - (s - \sigma)^2 + \ell(s)\right)$ is a well defined random variable.
\end{cor}

The definition of the Airy process given by \eqref{eq:detform} is not well adapted to studying variational problems (as it only deals with finite dimensional distributions). Let us note that \cite[Theorem 2]{Corwin-Quastel-Remenik11} provides a concise Fredholm determinant formula for $\Prob(\AiryA(s) \leq g(s) \text{ for }s\in [a,b])$, for any interval $[a,b]$ and any $g\in H^1([a,b])$ (i.e. both $g$ and its derivative are in $L^2([a,b])$). As we do not utilize this formula, we do not restate it here.

\subsection{Main technical tools}
The main technical tools in this paper are results stemming from the uniform slow decorrelation property that allows us to generalize Proposition \ref{prop:Johansson_weak} by Johansson, and the Gibbs property of a multilayer line ensemble extension of the LPP model. As this Gibbs property will require some explanation, we delay a discussion of it until Section \ref{sec:Gibbs}.

Recall the stochastic process $H_N(s)$ defined in \eqref{eq:H_N_def}. We define more generally
\begin{equation} \label{eq:defn_of_tilde_H}
  \tilde{H}_N(s) = \frac{1}{\bnought N^{1/3}} \bigg( \G \left( N + \ell_N(s) N^{\alpha} + s \cnought N^{2/3}, N + \ell_N(s) N^{\alpha} - s \cnought N^{2/3} \right) - \anought (N + \ell_N(s) N^{\alpha}) \bigg),
\end{equation}
where $\alpha \in [0, 1)$ is a parameter and $\ell_N(s)$ is a sequence of continuous functions such that the curve $L = (\ell_N(s)N^{\alpha} + s \cnought N^{2/3}, \ell_N(s)N^{\alpha} - s \cnought N^{2/3})$, ($s \in \realR$) is a down-right lattice path.

If  $\alpha = 0$ and $\ell_N(s) = l^0(s(\cnought N^{2/3})$ where $l^0(s)$ is defined in  \eqref{eq:defn_l^0}, then $\tilde{H}(s)$ is equal to $H_N(s)$ (up to an overall additive difference of $\ell_N(s)/(\bnought N^{1/3})$) defined in \eqref{eq:H_N_def}.

\begin{thm} \label{thm:uniform_slow_decorr}
  Let $\tilde{H}_N(s)$ be defined in \eqref{eq:defn_of_tilde_H} with $\alpha \in (0, 1)$ and $\ell_N(t)$ continuous on $[-M, M]$ and $\max_{s \in [-M, M]}\, \lvert \ell_N(s) \rvert < C$  for all large enough $N$. Then $\tilde{H}_N(s) - H_N(s)$ converges in probability to $0$ in $\mathcal{C}([-M, M],\realR)$, that is, given $\epsilon, \delta > 0$, there is an integer $N_0$ that depends only on $M$, $\alpha$ and $C$ such that
  \begin{equation}
  \label{eq:uniform_slow_decorrelation}
    \Prob \left( \max_{s \in [-M, M]} \lvert H_N(s) - \tilde{H}_N(s) \rvert \geq \delta \right) < \epsilon
  \end{equation}
  if $N > N_0$.
\end{thm}
The slow decorrelation property is a common feature in many models in the KPZ universality, including the LPP model, and equivalently the TASEP model, considered in this paper. As a pointwise property, it is studied first in \cite{Ferrari08} and then comprehensively in \cite{Corwin-Ferrari-Peche12}. Let $M \to 0_+$, then we have the result that as $N \to \infty$, $N^{-1/3} (\G(N, N) - \anought N)$ is equal to $N^{-1/3} (\G(N + \ell_N(0) N^{\alpha}, N + \ell_N(0) N^{\alpha}) - \anought(N + \ell_N(0)N^{\alpha}))$ in probability. This is a special case of the slow decorrelation result obtained in \cite{Corwin-Ferrari-Peche12}, where the characteristic line is the $\pi/4$ radial line. Theorem \ref{thm:uniform_slow_decorr} generalizes the pointwise slow decorrelation to be uniform on an interval.

Theorem \ref{thm:uniform_slow_decorr} gives control of $\tilde H_N(s)$ in any fixed interval $[-M,M]$. Outside this fixed interval we need the following lemma to control the point-to-curve LPPs by point-to-point LPPs as shown in Figure \ref{fig:Gibbs_path}. The lemma is a special consequence of the Gibbs property (see Section \ref{sec:Gibbs}), but it suffices for our paper.
\begin{lem} \label{lem:PNG_max}
  Suppose $N > 0$, $K_1 < K_2 < K_3$ are integers between $-N$ and $N$, and $M_1, M_2, M_3$ are real numbers such that $(K_1, M_1), (K_2, M_2), (K_3, M_3)$ are colinear, \ie,
  \begin{equation} \label{eq:colinear}
    \frac{M_1 - M_2}{K_1 - K_2} = \frac{M_2 - M_3}{K_2 - K_3}.
  \end{equation}
  Let $c \in (0, 1)$ be a constant and let $l^0(s)$ be defined in \eqref{eq:defn_l^0}. Then
  \begin{equation} \label{eq:Gibbs_thm}
    \begin{split}
      & \Prob \left( \max_{K_1 \leq s \leq K_2 - c(K_2 - K_1)} \G(N + l^0(s) + s, N + l^0(s) - s) \geq M_0 \right) \\
      \leq {}& (2 + \epsilon_{\min(c(K_2 - K_1), K_3 - K_2)}) \Prob(\G(N + K_2, N - K_2) \geq M_2) \\
      & + \Prob(\G(N + K_3, N - K_3) \leq M_3),
    \end{split}
  \end{equation}
  where for all $t > 0$, $\epsilon_t$ is a positive constant such that $\epsilon_t \to 0$ as $t \to \infty$.
\end{lem}
\begin{figure}[htb]
  \centering
  \includegraphics{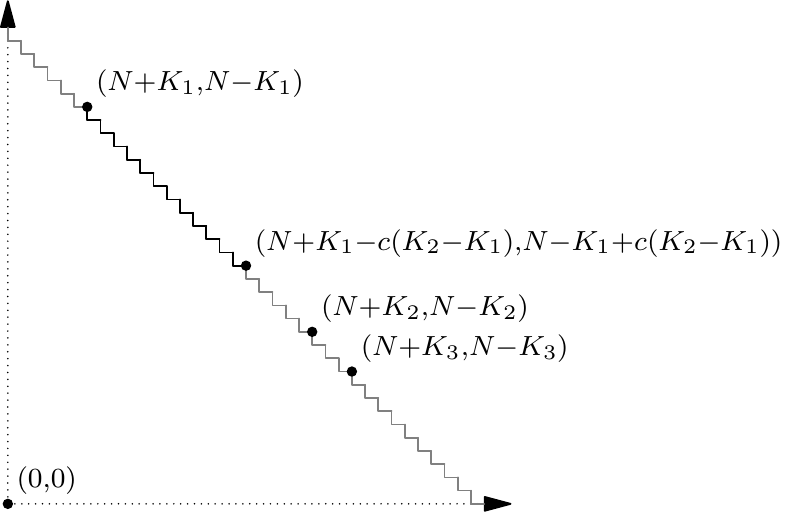}
  \caption{The points $(N + K_1, N - K_1)$, $(N + K_2, N - K_2)$ and $(N + K_3, N - K_3)$ are on the same diagonal down-right lattice path. The left-hand side of \eqref{eq:Gibbs_thm} is the LPP between the point $(0, 0)$ and the down-right lattice path between $(N + K_1, N - K_1)$ and $(N + K_1 - c(K_2 - K_1), N - K_1 + c(K_2 - K_1))$, shown in solid.}
  \label{fig:Gibbs_path}
\end{figure}

\section{Proof of Theorems \ref{thm:main}, \ref{thm:main_TASEP}, and \ref{thm:inhomogeneous}} \label{sec:proofs_of_main_theorems}

In this section, we give the detail of the proof of Theorem \ref{thm:main} in Section \ref{sec:proof_thm:main}, and show briefly that Theorem \ref{thm:main_TASEP} can be proved by the same method as  Theorem \ref{thm:main} in Section \ref{sec:proof_thm:TASEP}. The proof of Theorem \ref{thm:inhomogeneous}, (as well as the proof of Corollary \ref{cor:corrollaries_of_inhomogeneous_LPP}\ref{thm:double_spiked_border}) is by the same method with some adaptations, and we discuss them in Section \ref{sec:proof_lem:technical}.

\subsection{Proof of Theorem \ref{thm:main}} \label{sec:proof_thm:main}

By the translational invariance of the lattice, we can shift the point $(N + [\sigma c_0 N^{2/3}], N - [\sigma c_0 N^{2/3}])$ into $(N, N)$, and thus if we can prove Theorem \ref{thm:main} in the special case that $\sigma = 0$, the general case is proved by shifting the lattice. This is because the above shift applied to $L_N$ does not change the fact that it satisfies the hypotheses. Therefore, we only prove the $\sigma = 0$ case of Theorem \ref{thm:main} for notational simplicity.

Recall $\mathrm{Hyp}\big(C,c_1,c_2,c_3,a_{\infty},b_{\infty},\{m_N\}\big)$ given in Definition \ref{thmhypo}. Without loss of generality, we let the interval $I_N$ defined there be $[-N^{c_2}, N^{c_2}]$. By \eqref{eq:L_N_intersect_sector}, we only need to consider the curve $(L_N \cap (-\infty, N) \times (-\infty, N))$. We divide it into parts $L^{\micro}_N(M)$,  $L^{\meso, L}_N(M)$, $L^{\meso, R}_N(M)$, and $L^{\macro}_N$, where the first three depend on a constant $M > 0$, such that, recalling that $L^{\ct}_N$ defined in \eqref{eq:L^ct},
\begin{align}
  L^{\micro}_N(M) = {}& \{ (x, y) \in L^{\ct}_N \mid \lvert x - y \rvert \leq 2M c_0 N^{2/3} \}, \\
  L^{\meso, L}_N(M) = {}& \{ (x, y) \in L^{\ct}_N \setminus L^{\micro}_N(M) \mid x < 0 \}, \\
  L^{\meso, R}_N(M) = {}& \{ (x, y) \in L^{\ct}_N \setminus L^{\micro}_N(M) \mid x > 0 \}, \\
  L^{\macro}_N = {}& (L_N \cap (-\infty, N) \times (-\infty, N)) \setminus L^{\ct}_N.
\end{align}

In Subsection \ref{sec:micro_window}, we show that for any fixed $M > 0$ and $\epsilon > 0$,
\begin{equation} \label{eq:micro_estimate}
  \left\lvert \Prob \left( \frac{G_{(N, N)}(L^{\micro}_N(M)) - \anought N}{\dnought N^{1/3}} \leq x \right) - \Prob \left( \max_{s \in [-M, M]} \AiryA(s) - s^2 + \ell(s) \leq x \right) \right\rvert < \epsilon
\end{equation}
for all $N$ large enough, independent of the particular formula of $\ell(s)$. In Subsection \ref{sec:meso_window}, we show that for any fixed $\epsilon > 0$, there is an $M$ such that for all $N$ large enough, independent of the particular formula of $\ell(s)$,
\begin{equation} \label{eq:meso_estimate}
  \Prob \left( \frac{G_{(N, N)}(L^{\meso, *}_N(M)) - \anought N}{\dnought N^{1/3}} > x \right) < \epsilon, \quad \text{for $* = L$ or $R$},
\end{equation}
and for any fixed $\epsilon > 0$, for all $N$ large enough, independent of the particular formula of $\ell(s)$,
\begin{equation} \label{eq:macro_estimate}
  \Prob \left( \frac{G_{(N, N)}(L^{\macro}_N) - \anought N}{\dnought N^{1/3}} > x \right) < \epsilon.
\end{equation}
Thus by the three inequalities \eqref{eq:micro_estimate}, \eqref{eq:meso_estimate} and \eqref{eq:macro_estimate}, and the limit identity
\begin{equation} \label{eq:Airy_process_property_to_be_cited}
  \lim_{M \to \infty} \Prob \left( \max_{s \in [-M, M]} \left(\AiryA(s) - s^2 + \ell(s) \right)< x \right) = \Prob \left( \max_{s \in \realR}\left( \AiryA(s) - s^2 + \ell(s)\right) < x \right)
\end{equation}
that is a consequence of Proposition \ref{prop:Airy_property}\ref{enu:prop:Airy_property:b}, we prove the inequality \eqref{ea:ineq_main} of Theorem \ref{thm:main}.

\subsubsection{Microscopic estimate} \label{sec:micro_window}

In this subsection we prove that the inequality \eqref{eq:micro_estimate} holds for large enough $N$, where $M > 0$ and $\epsilon > 0$ is a constant.

%

Since $H_N(s)$ (recall from \eqref{eq:H_N_def}), as a stochastic process in $s \in [-M, M]$, converges weakly to $\AiryA(s) - s^2$ and $l_N(s)$ uniformly converges to $0$, we have that for any $\epsilon > 0$ there is a $\delta > 0$ such that for large enough $N$ independent of $\ell(s)$
\begin{align}
  \Prob \Big( \max_{s \in [-M, M]} \big(H_N(s) + (\ell(s) + l_N(s))\big) \leq x + \frac{\delta}{2} \Big) < {}& \Prob \Big( \max_{s \in [-M, M]} \big(\AiryA(s) - s^2 + \ell(s) \big)\leq x + \delta \Big) + \frac{\epsilon}{3}, \label{eq:L-P_metric:1} \\
  \Prob \Big( \max_{s \in [-M, M]}\big(\AiryA(s) - s^2 + \ell(s) \big)\leq x - \delta \Big) - \frac{\epsilon}{3} < {}& \Prob \Big( \max_{s \in [-M, M]} \big(H_N(s) + (\ell(s) + l_N(s)) \big)\leq x - \frac{\delta}{2} \Big). \label{eq:L-P_metric:2}
\end{align}
By Proposition \ref{prop:Airy_property}\ref{enu:prop:Airy_property:a}, the Airy process is locally Brownian, so if $\delta$ is small enough, then
\begin{align}
  \Prob \left( \max_{s \in [-M, M]} \left(\AiryA(s) - s^2 + \ell(s)\right) \leq x + \delta \right) - \Prob \left( \max_{s \in [-M, M]}\left( \AiryA(s) - s^2 + \ell(s)\right)\leq x \right) < {}& \frac{\epsilon}{3}, \label{eq:Airy_uniform_continuous:1} \\
  \Prob \left( \max_{s \in [-M, M]}\left( \AiryA(s) - s^2 + \ell(s)\right) \leq x \right) - \Prob \left( \max_{s \in [-M, M]}\left( \AiryA(s) - s^2 + \ell(s)\right)\leq  x - \delta \right) < {}& \frac{\epsilon}{3}. \label{eq:Airy_uniform_continuous:2}
\end{align}
The uniform slow decorrelation of LPP given in Theorem \ref{thm:uniform_slow_decorr} implies that
\begin{equation}\label{eq:micro_estimate1}
\Prob\left(\max_{s\in[-M,M]}\left\lvert \tilde H_N(s) - H_N(s)\right\rvert > \frac{\delta}{2} \right)<\frac{\epsilon}{3}
\end{equation}
for large enough $N$. Here $\tilde H_N(s)$ is defined as in \eqref{eq:defn_of_tilde_H} with $\ell_N(s)=\ell(s)+l_N(s)$ and $\alpha =1/3$. This bound relies on the fact that $\max_{s}|\ell_N(s)|$ is bounded.

The inequalities above yield that
\begin{equation} \label{eq:final_step_micro_window}
  \begin{split}
    \Prob \left( \frac{G_{(N,N)}(L^{\micro}_N(M))-\anought N}{\dnought N^{\frac{1}{3}}} \leq x \right) = {}& \Prob \left( \max_{s \in [-M, M]} \left(\tilde{H}_N(s) + (\ell(s) + l_N(s))\right) \leq x \right) \\
    < {}& \Prob \left( \max_{s \in [-M, M]}\left( H_N(s) + (\ell(s) + l_N(s))\right) \leq x + \frac{\delta}{2} \right) + \frac{\epsilon}{3} \\
    < {}& \Prob \left( \max_{s \in [-M, M]} \left(\AiryA(s) - s^2 + \ell(s)\right) \leq x + \delta \right) + \frac{2\epsilon}{3} \\
    < {}& \Prob \left( \max_{s \in [-M, M]}\left( \AiryA(s) - s^2 + \ell(s)\right) \leq x \right) + \epsilon.
  \end{split}
\end{equation}
Thus one direction of inequality \eqref{eq:micro_estimate} follows, and the proof of the other direction of \eqref{eq:micro_estimate} is similar. Finally note that the above inequalities only relied on the boundedness of $\max_s|\ell(s)|$ and not on the particular form of $\ell(s)$. This implies that the choice of $N_0$ for which the theorem holds can be made uniformly over all $\ell(s)$ satisfying the hypotheses.


\subsubsection{Macroscopic and mesoscopic estimates} \label{sec:meso_window}

\paragraph{Macroscopic estimate}

Inequality \eqref{eq:macro_estimate} is a direct consequence of Lemma \ref{lem:weaker_technical}. From \eqref{eq:region_D} it follows that for $(x, y)\in L^{\macro}_N$, $(N^{-1} x, N^{-1} y) \in D$. Also recall that $L_N$ satisfies the relation \eqref{eq:L_N_outside}. Thus, by Lemma \ref{lem:weaker_technical}, we have that for all $N$ large enough,
\begin{equation} \label{eq:ptp_macro}
  \Prob \left( \frac{G_{(N, N)}(x, y) - \anought N}{\dnought N^{1/3}} > x \right) < e^{-c N^{2 c_2}},
\end{equation}
where $c > 0$ depends on $c_3$ in \eqref{eq:L_N_outside} but not the shape of $L^{\macro}_N$. Note that there are fewer than $4(1 + q^{-1/2})^2 N^2$ lattice points (i.e. with integer coordinates) whose image under the scaling transform $(x, y) \to (N^{-1}x, N^{-1}y)$ lies in $D$. Thus we can pick $(x_i, y_i)$ on $L^{\macro}$ where $i = 1, \dotsc, [4(1 + q^{-1/2})^2 N^2]$, such that for all $(x, y) \in L^{\macro}$, $G_{(N, N)}(x, y)$ is equal to at least one $G_{(N, N)}(x_i, y_i)$. Thus
\begin{equation}
  \begin{split}
    \Prob \left( \frac{G_{(N, N)}(L^{\macro}_N) - a_0 N}{d_0 N^{1/3}} > x \right) < {}& \sum^{[4(1 + q^{-1/2})^2 N^2]}_{i = 1} \Prob \left( \frac{G_{(N, N)}(x_i, y_i) - a_0 N}{d_0 N^{1/3}} > x \right) \\
    < {}& 4(1 + q^{-1/2})^2 N^2 e^{-cN^{2c_2}},
  \end{split}
\end{equation}
and we obtain inequality \eqref{eq:macro_estimate} if $N$ is large enough.

\paragraph{Mesoscopic estimate}

By the symmetry of the lattice model, we only need to prove inequality \eqref{eq:meso_estimate} with $* = R$.

Before giving the proof, we remark that the simple approach in the macroscopic estimate fails in this case, since summing up all the point-to-point LPP between $(N, N)$ and lattice points on $L^{\meso, R}_N(M)$ gives a too large upper bound of the point-to-curve LPP $G_{(N, N)}(L^{\meso, R}_N)$. Before giving the technical proof, we explain the idea. We divide $L^{\meso, R}_N$ into segments according to the intervals $I(k)$ in \eqref{eq:defn_I(k)}. Then on each segment, we estimate the point-to-curve LPP (actually the upper bound $\Prob(k)$ defined in \eqref{eq:defn_Prob(k)}) by the point-to-point LPPs between $(0, 0)$ and the two points in \eqref{eq:first_point} and \eqref{eq:second_point}. We estimate the point-to-point LPPs by Lemma \ref{lem:weaker_technical}, and the relation between point-to-point LPPs and the point-to-curve LPP is established  by Lemma \ref{lem:PNG_max}.

Recall that $L^{\meso, R}_N(M) \subseteq L^{\ct}_N$ is defined in \eqref{eq:L^ct} by a continuous function $\ell(s) + l_N(s)$ for $s \in [M, N^{c_2}]$ , where $\ell(s)$ is bounded below by $C + c_1 s^2$ and $l_N(s)$ converges uniformly to $0$ as $N \to \infty$. By the inequality \eqref{eq:upper_bound_l}, we have that
\begin{equation} \label{eq:defin_tilde_M}
  \ell(s) < c'_1 s^2 \quad \text{for all $s \in [\tilde{M}, N^{c_2}]$, where $c'_1 \in (c_1, 1)$ and $\tilde{M} = \sqrt{C/(c'_1 - c_1)}$}.
\end{equation}
Taking
\begin{equation} \label{eq:c''_1_defn}
  c''_1 \in (1, \frac{2}{1 + c'_1}),
\end{equation}
since $x$ is a constant, it suffices to prove the inequality
\begin{equation} \label{eq:meso_intermediate}
  \Prob \left( G_{(N, N)}(L^{\meso, R}_N(M)) >  \anought N - c'_1(c''_1)^2 M^2 \dnought N^{1/3} \right) < \epsilon
\end{equation}
for all $M > \tilde{M}$ and large enough $N$.

For all $k = 0, 1, 2, \dotsc$ we denote
\begin{equation} \label{eq:defn_I(k)}
  c(k) = (c''_1)^k, \quad C_k = c'_1 (c(k)M)^2, \quad \text{and the interval} \quad I(k) = [c(k - 1)M, c(k)M],
\end{equation}
and define the down-right lattice paths
\begin{equation}
  L(k) = \{ (s \cnought N^{2/3} - l_0(s \cnought N^{2/3}) - [C_k \dnought N^{1/3}],\ -s \cnought N^{2/3} - l_0(s \cnought N^{2/3}) - [C_k \dnought N^{1/3}]) \mid s \in I(k) \}.
\end{equation}
Since on each $I(k)$, $\ell(s) < C_k$ as long as $\ell(s)$ is defined, and $c'_1 (c''_1)^2 M^2 < C_k$ for all $k$, it is clear that if we denote
\begin{equation} \label{eq:defn_Prob(k)}
  \Prob(k) = \Prob \left( (G_{(N, N)}(L(k)) \geq \anought N - C_k \dnought N^{1/3}) \right),
\end{equation}
then as $N$ is large enough,
\begin{equation} \label{eq:meso_intermediate_sum}
  \begin{split}
    & \Prob \left( G_{(N, N)}(L^{\meso, R}_N(M)) >  \anought N - c'_1 (c''_1)^2 M^2 \dnought N^{1/3} \right) \\
    \leq {}& \Prob \left( \max_{1 \leq k \leq [\log N^{c_2}/\log c_1]} (G_{(N, N)}(L(k)) \geq \anought N - C_k \dnought N^{1/3}) \right) \\
    \leq {}& \sum^{[\log N^{c_2}/\log c_1]}_{k = 1} \Prob(k).
  \end{split}
\end{equation}
To estimate $\Prob(k)$, we note that by the choice of $c''_1$ in \eqref{eq:c''_1_defn}, there exist $\delta_1, \delta_2, \delta_3, \delta_4 > 0$ such that $\delta_2 < \delta_3$ and the points
\begin{equation}
  (1, c'_1 (c''_1)^2), \quad (c''_1 + \delta_1, (1 - \delta_3) (c''_1 + \delta_1)^2), \quad (c''_1 + \delta_2, (1 + \delta_4) (c''_1 + \delta_2)^2)
\end{equation}
are collinear. Then by a simple affine transformation, the points
\begin{equation}
  \begin{gathered}
    \left( N + c(k - 1)M \cnought N^{2/3},\ \anought N - C_k \dnought N^{1/3} \right), \\
    \left( N + (c(k) + \delta_1 c(k - 1))M \cnought N^{2/3},\ \anought N - (1 - \delta_3)(1 + \delta_2/c''_1)^2 C_k \dnought N^{1/3} \right), \\
    \left( N + (c(k) + \delta_2 c(k - 1))M \cnought N^{2/3},\ \anought N - (1 + \delta_4)(1 + \delta_2/c''_1)^2 C_k \dnought N^{1/3} \right)
  \end{gathered}
\end{equation}
are collinear, as well as the three points
\begin{subequations}
  \begin{gather}
    \left( N + [c(k - 1)M \cnought N^{2/3}],\ \anought N - C_k \dnought N^{1/3} \right), \\
    \left( N + [(c(k) + \delta_1 c(k - 1))M \cnought N^{2/3}],\ \anought N - (1 - \delta_{3, N, k})(+ \delta_1/c''_2)^2 C_k \dnought N^{1/3} \right), \label{eq:first_point} \\
    \left( N + [(c(k) + \delta_2 c(k - 1))M \cnought N^{2/3}],\ \anought N - (1 + \delta_4)(c(k) + \delta_2)^2 M^2 d_0 N^{1/3} \right) \label{eq:second_point}
  \end{gather}
\end{subequations}
collinear, where $\delta_{3, N, k} \to \delta_3$ as $N \to \infty$ uniformly in $k$. We only need that for $N$ large enough
\begin{equation}
  \delta_{3, N, k} > \frac{\delta_3}{2}.
\end{equation}
Then by using the symmetry of the lattice and applying  Lemma \ref{lem:PNG_max}, we have
\begin{equation}
  \begin{split}
    \Prob(k) \leq {}& \Prob \left( \max^{[c(k)M \cnought N^{2/3} + 1]}_{s = [c(k - 1)M \cnought N^{2/3}]} \G(N + [C_k \dnought N^{1/3}] + s, N + [C_k \dnought N^{1/3}] - s) \geq \anought N - c'_1(c(k) M)^2 \dnought N^{1/3} \right) \\
    \leq {}& (2 + \epsilon_{\min(\delta_1, \delta_2 - \delta_1) \cdot c(k - 1)M \cnought N^{2/3}}) \Prob \bigg( \G \left( N + [C_k \dnought N^{1/3}] + [(c(k) + \delta_1 c(k - 1))M \cnought N^{2/3}], \right. \\
    & \left.  N + [C_k \dnought N^{1/3}] + [(c(k) + \delta_1 c(k - 1))M \cnought N^{2/3}] \right) \geq \anought N - (1 - \delta_{3, N, k}) (1 + \delta_1/c''_2)^2 C_k \dnought N^{1/3} \bigg) \\
    & + \Prob \bigg( \G \left( N + [C_k \dnought N^{1/3}] + [(c(k) + \delta_2 c(k - 1))M \cnought N^{2/3}], \right. \\
    & \left.  N + [C_k \dnought N^{1/3}] + [(c(k) + \delta_2 c(k - 1))M \cnought N^{2/3}] \right) \leq \anought N - (1 + \delta_4)(c(k) + \delta_2/c''_2)^2 C_k \dnought N^{1/3} \bigg),
  \end{split}
\end{equation}
where the term $\epsilon_{\min(\delta_1, \delta_2 - \delta_1) \cdot c(k - 1)M \cnought N^{\frac{2}{3}}}$ is defined in Lemma \ref{lem:PNG_max} and vanishes as $N \to \infty$. An application of Lemma \ref{lem:weaker_technical}, shows that
\begin{equation}
  \Prob(k) < e^{-Mk}
\end{equation}
for large enough $M$. Thus \eqref{eq:meso_intermediate} is proved by taking the sum of $\Prob(k)$ in \eqref{eq:meso_intermediate_sum}.

\subsection{Proof of Theorem \ref{thm:main_TASEP}} \label{sec:proof_thm:TASEP}

  By the relation \eqref{eq:relation_TASEP_LPP}, we have that under the assumption that $2\dnought^* N^{1/3} x$ and $2\sigma \cnought N^{2/3}$ are integers with the same parity,
  \begin{equation}
    \Prob \left( \frac{h_N(2\sigma \cnought N^{\frac{2}{3}}; \anought^* N) - 2N}{2\dnought^* N^{\frac{1}{3}}} > -x \right) = \Prob \left(  G^*_{(N + \sigma \cnought N^{\frac{2}{3}} - \dnought^* N^{\frac{1}{3}} x, N - \sigma \cnought N^{\frac{2}{3}} - \dnought^* N^{\frac{1}{3}} x)}(L^*_N) \leq [\anought^* N] \right).
  \end{equation}
   The above equation implies that the result of Theorem \ref{thm:main_TASEP} amounts to computing the $N \to \infty$ limit of
\begin{equation}
  \Prob \left(  G^*_{(N + \sigma \cnought N^{2/3} - \dnought^* N^{1/3} x, N - \sigma \cnought N^{2/3} - \dnought^* N^{1/3} x)}(L^*_N) \leq [\anought^* N] \right).
\end{equation}
This readily follows from the relation between $G$ and $G^*$ as well as Theorem \ref{thm:main}.

\subsection{Proof of Theorem \ref{thm:inhomogeneous}} \label{sec:proof_lem:technical}

For the proof of the first part of Theorem \ref{thm:inhomogeneous}, we assume $D_N = \realR$ without loss of generality. We express
\begin{equation}
  G^{f_N}_{(N, N)}(L) = \max\left( G^{f_N, \micro}_{(N, N)}(L), G^{f_N, \meso}_{(N, N)}(L), G^{f_N, \macro}_{(N, N)}(L) \right),
\end{equation}
where
\begin{equation}
  G^{f_N, *}_{(N, N)}(L) = \max_{y \in I_*} \,\left( G_{(N, N)}(0, y) - f_N(y)\right), \quad * = \micro, \meso \text{ or } \macro,
\end{equation}
and, letting $M > 0$,
\begin{equation}
  I_* =
  \begin{cases}
    [-M 2\cnought N^{2/3}, M 2\cnought N^{2/3}] & \text{for $* = \micro$}, \\
    [-2\cnought N^{2/3 + c_2}, 2\cnought N^{2/3 + c_2}] \setminus I_{\micro} & \text{for $* = \meso$}, \\
    (-\infty, N] \setminus (I_{\micro} \cup I_{\meso}) & \text{for $* = \macro$}.
  \end{cases}
\end{equation}
Similar to the proof of Theorem \ref{thm:main} in Section \ref{sec:proof_thm:main}, we show that for any fixed $M$ and $\epsilon > 0$, there exists $N_0$ such that for all $N>N_0$,
\begin{equation} \label{eq:micro_estimate_f_N}
  \left\lvert \Prob \left( \frac{G^{f_N, \micro}_{(N, N)}(L) - \anought N}{\bnought N^{1/3}} \leq x \right) - \Prob \left( \max_{s \in [-M, M]} \left(\AiryA(s) - s^2 + \ell(s) \right)\leq x \right) \right\rvert < \epsilon.
\end{equation}
Here, and in what follows, the choice of $N_0$ can be seen to only depend on the parameters $C,c_1,c_2,c_3,a_{\infty},b_{\infty}$ in $\mathrm{Hyp}^{\vert}\big(C,c_1,c_2,c_3,a_{\infty},b_{\infty},\{m_N\}\big)$ and not on the particular form of $f_N$. Then we show that for any fixed $\epsilon > 0$, for all $N>N_0$,
\begin{equation} \label{eq:meso_estimate_f_N}
  \Prob \left( \frac{G^{f_N, \meso}_{(N, N)}(L) - \anought N}{\bnought N^{1/3}} > x \right) < \epsilon,
\end{equation}
and at last show that for any fixed $\epsilon > 0$, for all $N>N_0$,
\begin{equation} \label{eq:macro_estimate_f_N}
  \Prob \left( \frac{G^{f_N, \macro}_{(N, N)}(L) - \anought N}{\bnought N^{1/3}} > x \right) < \epsilon.
\end{equation}
Thus we have proved Theorem \ref{thm:inhomogeneous}.

We turn now to prove \eqref{eq:micro_estimate_f_N}. Recall the relation between $f_N$ and $\ell(s) + l_N(s)$ defined in \eqref{eq:f_N_char} of $\mathrm{Hyp}^{\vert}\big(C,c_1,c_2,c_3,a_{\infty},b_{\infty},\{m_N\}\big)$. If follows that
\begin{equation}
  \Prob \bigg( \frac{G^{f_N, \micro}_{(N, N)}(L) - \anought N}{\bnought N^{1/3}} \leq x \bigg) = \Prob \Big( \max_{s \in [-M, M]} \big(\tilde{H}_N(s) + \ell(s) + l_N(s)\big) \leq x \Big),
\end{equation}
where $\tilde{H}_N(s)$ is defined in \eqref{eq:defn_of_tilde_H}, with $\alpha = 2/3$ and $\ell_N(s) = \cnought s$. Note that here $\ell_N(s)$ defines the shape of the straight line $L$, so it is a linear function. Unlike $\ell_N(s)$ in \eqref{eq:micro_estimate1}, where $\ell_N(s) = \ell(s) + l_N(s)$, here $\ell(s) + l_N(s)$ defines $f_N$ but not the shape of $L$ or the function $\ell_N(s)$. Then using the convergence results in Theorem \ref{thm:uniform_slow_decorr} and Proposition \ref{prop:Johansson_weak}, we arrive at \eqref{eq:micro_estimate_f_N} by an argument similar to that of Section \ref{sec:micro_window}.

To prove \eqref{eq:meso_estimate_f_N}, we use a simple inequality that for any lattice points $(x_0, y_0)$, $(x, y)$ and $(x', y')$ such that $x_0 \geq x \geq x'$ and $y_0 \geq y \geq y'$, we have
\begin{equation}
  G_{(x_0, y_0)}(x, y) \leq G_{(x_0, y_0)}(x', y') - G_{(x, y)}(x', y').
\end{equation}
Now we take $(x_0, y_0) = (N, N)$, $(x, y) = (0, s)$ where the integer $s \in I_{\meso}$ and corresponding to $(x, y)$, with the same $s$,
\begin{equation}
  (x', y') =
  \begin{cases}
    (-[M \cnought N^{2/3}] - \frac{s}{2},\ -[M \cnought N^{2/3}] + \frac{s}{2}) & \text{if $s$ is even}, \\
    (-[M \cnought N^{2/3}] - \frac{s - 1}{2},\ -[M \cnought N^{2/3}] + \frac{s + 1}{2}) & \text{if $s$ is odd}.
  \end{cases}
\end{equation}
It is easy to see that if we prove that if $M$ is large enough, then for all large enough $N$,
\begin{equation} \label{eq:Gibbs_for_f_N}
  \Prob \left( \max_{s \in \intZ \text{ and } s \in I_{\meso}} G_{(N, N)}(x', y') \geq \anought N + \anought M \cnought N^{\frac{2}{3}} + x \bnought N^{1/3} \right) < \frac{\epsilon}{2},
\end{equation}
and uniformly for all $s \in \intZ \cap I_{\meso}$, if $N$ is large enough,
\begin{equation} \label{eq:ptp_f_N}
  \Prob \left( G_{(x, y)}(x', y') \leq \anought M \cnought N^{2/3} - f_N \left( \frac{s}{2\cnought N^{2/3}} \right) \right) < \frac{\epsilon}{2} \frac{1}{4\cnought N^{2/3 + c_2}},
\end{equation}
then \eqref{eq:meso_estimate_f_N} is proved.

The inequality \eqref{eq:Gibbs_for_f_N} is analogous to \eqref{eq:meso_intermediate} and can be proved by the arguments used in Section \ref{sec:meso_window}. The inequality \eqref{eq:ptp_f_N} is a direct consequence of Proposition \ref{prop:previous_pt_to_pt}\ref{enu:prop:previous_pt_to_pt_2}. Then the proof of Theorem \ref{thm:inhomogeneous} is complete.

To prove \eqref{eq:macro_estimate_f_N}, we estimate the probability that the point-to-point LPP $\Prob(G_{(N, N)}(0, s) - f(N) > \bnought N^{1/3} x)$ by Lemma \ref{lem:weaker_technical} for all $s \in \intZ \cap I_{\macro}$, and then sum up all these probabilities as an upper bound of the left-hand side of \eqref{eq:macro_estimate_f_N}. The argument is similar to the proof of \eqref{eq:macro_estimate} and the detail is omitted.

The proof of the second part of Theorem \ref{thm:inhomogeneous} is similar. We divide the $\llcorner$-shaped path $\tilde{L}$ defined in \eqref{eq:L_shaped_L_tilde} into the ``micro'', ``meso'' and ``macro'' parts according to the distance to the corner $(0, 0)$, and use the three methods to estimate the point-to-curve LPP between $(N, N)$ to them, as above. We omit the details.

\section{Proofs of Corollaries \ref{thm:Bernoulli_initial_condition} and \ref{cor:corrollaries_of_inhomogeneous_LPP}} \label{sec:Proofs_of_corollaries}

\subsection{Proof of Corollary \ref{cor:corrollaries_of_inhomogeneous_LPP}\ref{thm:spiked_border} and \ref{thm:double_spiked_border}} \label{subsec:proof_spiked_border}

Parts \ref{thm:spiked_border} and \ref{thm:double_spiked_border} of Corollary \ref{cor:corrollaries_of_inhomogeneous_LPP} are direct consequences of the first and second parts of Theorem \ref{thm:inhomogeneous}, respectively. We only give detail of the proof of part \ref{thm:spiked_border}, since that of part \ref{thm:double_spiked_border} is similar.

Define the random function $f(x)$ on the domain $D = [0, \infty)$ by
\begin{equation}
  f(x) = -\G(k - 1, x)
\end{equation}
where the weight on the lattice is assumed to be inhomogeneous and the weights $w(i, j)$ with $i = 1, \dotsc, k$ are specified by \eqref{eq:weight_k_borders}. Then the point-to-point LPP $\G(N, N)$ is expressed as
\begin{equation}
  \G(N, N) = \max_{x \in [0, N]} G_{(N, N)}(k, x) - f(x).
\end{equation}
We see that the $\G(N, N)$ on the lattice with inhomogeneous weights has the same distribution as $G^{f}_{(N - k, N)}(L)$, where the notation is the same as in Theorem \ref{thm:inhomogeneous}. As $N \to \infty$, we have ($\tilde{G}$ is defined in \eqref{eq:defn_tilde_G})
\begin{equation} \label{eq:approx_by_square}
  \lim_{N \to \infty} \Prob (\tilde{G}_N \leq x) = \lim_{N \to \infty} \Prob \left( \frac{G^f_{(N, N)}(L) - \anought N}{\bnought N^{1/3}} \leq x \right).
\end{equation}
Although the random function $f(x)$ is not in the form of $f_N(x)$ in \eqref{eq:f_N_char}, the difference is only a constant term. We write for any $N$
\begin{equation}
  f(2s\cnought N^{\frac{2}{3}}) = -s\anought\cnought N^{\frac{2}{3}} - \ell_N(s)\dnought N^{\frac{1}{3}}.
\end{equation}
For any $\epsilon > 0$, by choosing the constant $C$ properly, the inequality
\begin{equation} \label{eq:estimate_f_growth}
  \ell_N(s) < C + \frac{1}{2} s^2
\end{equation}
is satisfied with probability at least $1 - \epsilon$. This is because $k$ is fixed and $\ell_N(s)$ behaves like the maximum of a $k$-particle Dyson Brownian motion which can easily be bounded by quadratic growth in time. So by Theorem \ref{thm:inhomogeneous}, given any $\epsilon > 0$, for large enough $N$
\begin{equation} \label{eq:cor_by_thm}
  \left\lvert \Prob \left( \frac{G^f_{(N, N)}(L) - \anought N}{\bnought N^{1/3}} \leq x \right) - \Prob \left( \max_{s \in (0, \infty)} \left( \AiryA(s) - s^2 + \ell_N(s) \right) \leq x \right) \right\rvert < \epsilon.
\end{equation}
 Furthermore, it is not hard to see that the random function $\ell_N(s)$ converges weakly to
\begin{equation}
\max_{0=s_0\leq s_1\leq\cdots\leq s_k=s}  \sqrt{2} \sum^k_{i = 1} \left(\B_i(s_i) - \B_i(s_{i - 1})\right) - 4\sum^k_{i = 1} w_i(s_i - s_{i - 1}) - s^2_k
\end{equation}
on any compact interval. At last, the weak convergence of $l^{(N)}(s)$, together with the estimate \eqref{eq:estimate_f_growth} and Proposition \ref{prop:Airy_property}\ref{enu:prop:Airy_property:b}, implies that
\begin{multline} \label{eq:final_step_BM}
  \lim_{N \to \infty} \Prob \left( \max_{s \in (0, \infty)} \left( \AiryA(s) - s^2 + \ell_N(s) \right) \leq x \right) = \\
  \Prob \left( \max_{0 = s_0 \leq s_1 \leq \dotsb \leq s_k \leq M} \left(\AiryA(s_k) + \sqrt{2} \sum^k_{i = 1} \left(\B_i(s_i) - \B_i(s_{i - 1})\right) - 4\sum^k_{i = 1} w_i(s_i - s_{i - 1}) - s^2_k\right) \leq x \right).
\end{multline}
Combining \eqref{eq:approx_by_square}, \eqref{eq:cor_by_thm} and \eqref{eq:final_step_BM}, we prove Corollary \ref{cor:corrollaries_of_inhomogeneous_LPP}\ref{thm:spiked_border}.

\subsection{Proof of Corollary \ref{cor:corrollaries_of_inhomogeneous_LPP}\ref{thm:inner_spiked_border}}

Let the weights $w(i, j)$ be defined as in Corollary \ref{cor:corrollaries_of_inhomogeneous_LPP}\ref{thm:inner_spiked_border}. Define the stochastic processes $B_{1, N}, \dotsc, B_{k, N}$ as
\begin{equation}
  B_{i, N}(s) =
  \begin{cases}
    \frac{1}{\bnought N^{1/3}} \left( G_{([\alpha N] + i, [\alpha N])}([\alpha N] + i, \alpha N + 2\cnought N^{\frac{2}{3}} s) - \anought\cnought N^{2/3} s \right) & \text{if $s \geq 0$}, \\
    \frac{1}{\bnought N^{1/3}} \left( -\G_{([\alpha N] + i, [\alpha N])}([\alpha N] + i, \alpha N + 2\cnought N^{\frac{2}{3}} s) - \anought\cnought N^{2/3} s \right) & \text{if $s < 0$}.
  \end{cases}
\end{equation}
Then we have the weak convergence
\begin{equation} \label{eq:weak*_middle_collumns}
  B_{i, N}(s) \Rightarrow \sqrt{2}\B_i(s) + 4w_i s
\end{equation}
on any compact interval as $N \to \infty$, where $\B_1(s), \dotsc, \B_k(s)$ are independent two-sided Brownian motions.

Next define the stochastic processes
\begin{align}
  A^{(1)}_N(s) = {}& \frac{1}{\bnought N^{1/3}} \left( \G([\alpha N], [\alpha N] - 2\cnought N^{2/3} s) - \anought(\alpha N - \cnought N^{2/3} s) \right), \\
  A^{(2)}_N(s) = {}& \frac{1}{\bnought N^{1/3}} \left( G_{(N, N)}([\alpha N] + k + 1, [\alpha N] + 2\cnought N^{2/3} s) - \anought(\alpha N - \cnought N^{2/3} s) \right),  \label{eq:weak*_converg_A^(2)_N}
\end{align}
By Theorem \ref{thm:uniform_slow_decorr} and Proposition \ref{prop:Johansson_weak}, we have the weak convergence that on any interval $[-M, M]$ as $N \to \infty$
\begin{equation} \label{eq:two_Airy_convergence_weak*}
  A^{(1)}_N(s) \Rightarrow \alpha^{1/3} \AiryA^{(1)}(\alpha^{-2/3} s) - \frac{s^2}{\alpha}, \quad A^{(2)}_N(s) \Rightarrow (1 - \alpha)^{1/3} \AiryA^{(2)}((1 - \alpha)^{-2/3} s) - \frac{s^2}{1 - \alpha},
\end{equation}
where $\AiryA^{(1)}(s)$ and $\AiryA^{(2)}(s)$ are two independent Airy processes.

We denote the three regions of $\realR^{k + 1}$
\begin{equation}
  \begin{split}
    R_1(M) = {}& \{ (s_0, s_1, \dotsc, s_k) \mid -M \leq s_0 \leq s_1 \leq \dotsb \leq s_k \leq M \}, \\
    R_2(M) = {}& \{ (s_0, s_1, \dotsc, s_k) \mid   s_0 \leq s_1 \leq \dotsb \leq s_k \leq M \text{ and } s_0 < -M \}, \\
    R_2(M) = {}& \{ (s_0, s_1, \dotsc, s_k) \mid  -M \leq s_0 \leq s_1 \leq \dotsb \leq s_k \text{ and } s_k > M \},
  \end{split}
\end{equation}
and write
\begin{equation}
  \frac{\G(N, N) - \anought N}{\bnought N^{1/3}} = \max \left( G^{(1)}_N(M), G^{(2)}_N(M), G^{(3)}_N(M) \right),
\end{equation}
where for $i = 1, 2, 3$,
\begin{equation}
  \begin{split}
    G^{(i)}_N(M) = {}& \frac{1}{\bnought N^{1/3}} \max_{(s_0, \dotsc, s_k) \in R_i(M)} \left( \vphantom{\sum^k_{i = 1}} \G([\alpha N], [\alpha N] + [2\cnought N^{2/3} s_0]) \right. \\
    & + \sum^k_{i = 1} \G_{([\alpha N] + i, [2\cnought N^{2/3} s_{i - 1}])}([\alpha N] + i, [2\cnought N^{2/3} s_i]) \\
    & \left. \vphantom{\sum^k_{i = 1}} + G_{(N, N)}([\alpha N] + k + 1, [2\cnought N^{2/3} s_k]) - \anought N \right) \\
    = {}& \max_{(s_0, \dotsc, s_k) \in R_i(M)} \left( A^{(1)}_N\left(\frac{ [2\cnought N^{2/3} s_0]}{2\cnought N^{2/3}}\right) + A^{(2)}_N\left(\frac{ [2\cnought N^{2/3} s_k]}{2\cnought N^{2/3}} \right) \right. \\
      & + \left. \sum^k_{i = 1} \left(\tilde{B}_{i, N}\left(\frac{ [2\cnought N^{2/3} s_i]}{2\cnought N^{2/3}}\right) - \tilde{B}_{i, N}\left(\frac{ [2\cnought N^{2/3} s_{i-1}]-1}{2\cnought N^{2/3}}\right)\right) \right).
  \end{split}
\end{equation}

It is a direct consequence of the convergence results \eqref{eq:weak*_middle_collumns} and \eqref{eq:two_Airy_convergence_weak*} that for any $M > 0$
\begin{multline} \label{eq:estimate_G^(6)}
  \lim_{N \to \infty} \Prob \left( G^{(1)}_N(M) \leq x \right) =
  \Prob \left( \max_{-M \leq s_0 \leq s_1 \leq \dotsb \leq s_k \leq M}\left (\alpha^{1/3} \AiryA^{(1)}(\alpha^{-2/3} s_0) + \sqrt{2} \sum^k_{i = 1} \left(\B_i(s_i) - \B_i(s_{i - 1})\right) \right.\right. \\
  +\left. ( 1 - \alpha)^{1/3} \AiryA^{(2)}((1 - \beta)^{-2/3} s_k)
  \left. \vphantom{\max_{s_0 \leq s_1 \leq \dotsb \leq s_k}\sum^k_{i = 1}} - 4\sum^k_{i = 1} w_i(s_i - s_{i - 1}) - \frac{s^2_0}{\alpha} - \frac{s^2_k}{1 - \alpha}\right) \leq x \right).
\end{multline}

To estimate $G^{(2)}_N(M)$, we recall the hypothesis $\mathrm{Hyp}^{\vert}\big(C,c_1,c_2,c_3,a_{\infty},b_{\infty},\{m_N\}\big)$ in Definition \ref{inhomogeneous_hypo} for Theorem \ref{thm:inhomogeneous}, and define analogously the function (cf.~\eqref{eq:outer_f_N^max} with $c_1 = 1/2$ and $c_3 = 1/100$)
\begin{equation}
  L_N(x) = N \max \left( \anought - \frac{\anought x/N}{2} - \frac{1}{2} \dnought \left( \frac{x/N}{2\cnought} \right)^2,\ a_0 \left( 1 - \frac{x}{N} \right) + \frac{1}{100} \left\lvert \frac{x}{N} \right\rvert \right).
\end{equation}
For a constant $K$ (to be chosen suitably large in what follows) also define the function (cf.~\eqref{eq:upper_bound_l} with $C = K$ and $c_1 = 1/2$)
\begin{equation}
  \ell^{\max}_N(s) = \max \left( K + \frac{s^2}{2}, \frac{1}{\dnought N^{\frac{1}{3}}} \left( \anought N - s\anought\cnought N^{\frac{2}{3}} - L_N(2s\cnought N^{\frac{2}{3}}) \right) \right).
\end{equation}
Then we have
$
  \Prob \left( G^{(2)}_N(M) \geq x \right) \leq \Prob_1 + \Prob_2 + \Prob_3 + \Prob_4 + \Prob_5,
$
where
\begingroup
\allowdisplaybreaks
\begin{align}
  \Prob_1 = {}& \Prob \bigg( \max_{s_0 < -M} A^{(1)}_N\Big(\frac{[2\cnought N^{2/3}s_0]}{2\cnought N^{2/3}}\Big) \geq -\alpha^{1/3} \ell^{\max}_N(\alpha^{-2/3} s_0) \bigg), \\
  \Prob_2 = {}& \Prob \bigg( \max_{s_k \in \realR \setminus (-M, M)} A^{(2)}_N\Big(\frac{[2\cnought N^{2/3}s_k]}{2\cnought N^{2/3}}\Big) \geq -(1 - \alpha)^{1/3} \ell^{\max}_N((1 - \alpha)^{-2/3} s_k) \bigg), \\
  \Prob_3 = {}& \Prob \bigg( \max_{s_k \in [-M, M]} A^{(2)}_N\Big(\frac{[2\cnought N^{2/3}s_k]}{2\cnought N^{2/3}}\Big) \geq K \bigg), \\
  \Prob_4 = {}& \Prob \bigg( \max_{\substack{s_0 \leq \dotsb \leq s_k \\ s_0 < -M, \\ s_k \in [-M, M]}} \sum^k_{i = 1} \Big(\tilde{B}_{i, N}\Big(\frac{[2\cnought N^{2/3}s_i]}{2\cnought N^{2/3}}\Big) - \tilde{B}_{i, N}\Big(\frac{[2\cnought N^{2/3}s_{i-1}]-1}{2\cnought N^{2/3}}\Big)\Big) \bigg. \notag \\
  & \bigg. \phantom{\max_{\substack{s_0 \leq \dotsb \leq s_k \\ s_0 < -M, \\ s_k \in [-M, M]}} \sum^k_{i = 1}} \geq \alpha^{1/3} \ell^{\max}_N(\alpha^{-2/3} s_0) - K + x \bigg), \\
  \Prob_5 = {}& \Prob \bigg( \max_{\substack{s_0 \leq \dotsb \leq s_k \\ s_0 < -M, \\ \lvert s_k \rvert \in [M, (1 - \alpha)N/(2c_0 N^{2/3})]}}  \sum^k_{i = 1} \Big(\tilde{B}_{i, N}\Big(\frac{[2\cnought N^{2/3}s_i]}{2\cnought N^{2/3}}\Big) - \tilde{B}_{i, N}\Big(\frac{[2\cnought N^{2/3}s_{i-1}]-1}{2\cnought N^{2/3}}\Big)\Big) \bigg. \notag \\*
  & \bigg.  \phantom{\max_{\substack{s_0 \leq s_1 \leq \dotsb \leq s_k \\ s_0 < -M, \\ \lvert s_k \rvert \in [M, (1 - \alpha)N/(2c_0 N^{2/3})]}}}\ge \alpha^{1/3} \ell^{\max}_N(\alpha^{-2/3} s_0) + (1 - \alpha)^{1/3} \ell^{\max}_N((1 - \alpha)^{-2/3} s_k) + x \bigg).
\end{align}%
\endgroup
Now we assume $\epsilon > 0$ is a small constant. As in the proof of Theorem \ref{thm:inhomogeneous}, we have that if $M$ is large enough, then for all large enough $N$,
\begin{equation}
  \Prob_1 < \epsilon, \quad  \Prob_2 < \epsilon.
\end{equation}
By the property of the Airy process in Lemma \ref{prop:Airy_property} and the convergence \eqref{eq:two_Airy_convergence_weak*} of $A_N^{(2)}$ to the Airy process, we have that for all $M>0$ there exists an $K > 0$ depending on $\epsilon$ such that the inequality
\begin{equation} \label{eq:determination_of_M'}
  \Prob_3 < \epsilon
\end{equation}
holds. In fact, since we have just argued that $\Prob_2 < \epsilon$ for $M$ large enough, it follows that once $M$ is large enough, $K$ can be chosen independent of $M$ (in fact the same $K$ can then be used for $M$ small as well by obvious containment of sets).
By a standard argument for random walks, we find that if $K$ depends on $\epsilon$ as in \eqref{eq:determination_of_M'} but not $M$, and $M$ is large enough, then for all $N$ large enough,
\begin{equation}
  \Prob_4 < \epsilon, \quad \text{and} \quad \Prob_5 < \epsilon.
\end{equation}
Hence we conclude that if $M$ is large enough, then for all $N$ large enough,
\begin{equation} \label{eq:estimate_G^(7)}
  \Prob \left( G^{(2)}_N(M) \geq x \right) < 5\epsilon.
\end{equation}
By a parallel argument, we have that if $M$ is large enough, then for all $N$ large enough,
\begin{equation} \label{eq:estimate_G^(8)}
  \Prob \left( G^{(3)}_N(M) \geq x \right) < 5\epsilon.
\end{equation}
Finally, by \eqref{eq:estimate_G^(6)}, \eqref{eq:estimate_G^(7)} and \eqref{eq:estimate_G^(8)}, together with Proposition \ref{prop:previous_pt_to_pt}\ref{enu:prop:previous_pt_to_pt_2},
\begin{equation}
\begin{split}
  &\lim_{M \to \infty} \Prob \left( \max_{-M \leq s_0 \leq s_1 \leq \dotsb \leq s_k \leq M} \left(\alpha^{1/3} \AiryA^{(1)}(\alpha^{-2/3} s_0) + \sqrt{2} \sum^k_{i = 1} \B_i(s_i) - \B_i(s_{i - 1}) \right. \right.\\
 & \qquad\qquad\qquad + ( 1 - \alpha)^{1/3} \AiryA^{(2)}((1 - \beta)^{-2/3} s_k)
  \left. \left.\vphantom{\max_{s_0 \leq s_1 \leq \dotsb \leq s_k}\sum^k_{i = 1}} - 4\sum^k_{i = 1} w_i(s_i - s_{i - 1}) - \frac{s^2_0}{\alpha} - \frac{s^2_k}{1 - \alpha}\right) \leq x \right)  \\
 &= \Prob \left( \max_{-\infty \leq s_0 \leq s_1 \leq \dotsb \leq s_k \leq \infty}\left (\alpha^{\frac{1}{3}} \AiryA^{(1)}(\alpha^{-\frac{2}{3}} s_0) + \sqrt{2} \sum^k_{i = 1} (\B_i(s_i) - \B_i(s_{i - 1}) )\right.\right. \\
 & \qquad\qquad\qquad\left.+ ( 1 - \alpha)^{1/3} \AiryA^{(2)}((1 - \beta)^{-2/3} s_k)
  \left. \vphantom{\max_{s_0 \leq s_1 \leq \dotsb \leq s_k}\sum^k_{i = 1}} - 4\sum^k_{i = 1} w_i(s_i - s_{i - 1}) - \frac{s^2_0}{\alpha} - \frac{s^2_k}{1 - \alpha}\right) \leq x \right),
  \end{split}
\end{equation}
we prove part \ref{thm:inner_spiked_border} of Corollary \ref{cor:corrollaries_of_inhomogeneous_LPP}.

\subsection{Proof of Corollary \ref{thm:Bernoulli_initial_condition}}

Since all the five parts of the corollary are similar, we only prove part \ref{enu:thm:Bernoulli_initial_condition:d} as the proofs of the other four parts are analogous or easier.

The random function $h^{\flatBernoulli}$ in \eqref{eq:initial_flat_Bernoulli} defines a random polygonal chain $L^{\flatBernoulli}$ by \eqref{eq:L_by_h}, we define a function $\ell_N(s)$ associated to it by the relation
\begin{equation}
  L^{\flatBernoulli} = \Big\{ \big(s\cnought N^{2/3} - \ell_N(s) \dnought^*N^{1/3}, -s\cnought N^{2/3} - \ell_N(s)\dnought^*N^{1/3}\big) \big\vert   s \in \realR \Big\}.
\end{equation}
Then $\ell_N(s)$ is a continuous function such that it is deterministic for $s < 0$ and random for $s > 0$. It is clear that for $s > 0$, $\ell_N(s)$ is mapped to the path of a simple symmetric random walk, such that
\begin{equation} \label{eq:corresp_RW}
  \left( 2\dnought^* N^{1/3} \ell_N\left(\frac{k}{2\cnought N^{2/3}}\right),\ k = 1, 2, \dotsc \right) \sim \left( \sum^k_{i = 1} X_i,\ k = 1, 2, \dotsc \right),
\end{equation}
where $X_i$ are in \iid\ distribution with $\Prob(X_i = -1) = \Prob(X_i = 1) = 1/2$.

Now we consider $\mathrm{Hyp}^*\big(C,c_1,c_2,c_3^*,a_{\infty},b_{\infty},\{m_N\}\big)$ defined in Definition \ref{thmhypo} and let $c_1 = 1/2,$, $c_2 = 1/6$, $c^*_3 = -1/100$, $a_{\infty} = -\infty$, $b_{\infty} = +\infty$, $m_N = 0$. We claim that for any $\epsilon > 0$, there is a large enough constant $C_{\epsilon}$ such that if we let $C = C_{\epsilon}$, then if $N$ is large enough, with probability greater than $1 - \epsilon$, $\mathrm{Hyp}^*\big(C,c_1,c_2,c_3^*,a_{\infty},b_{\infty},\{m_N\}\big)$ is satisfied by $\ell(s) = \ell_N(s)$, $l_N(s) = 0$, and $L^*_N = L^{\flatBernoulli}$. To check it, we note that for $s \leq 0$, $\ell_N(s)$ is a deterministic function whose value is close to $0$, and on the ``flat'' part of $L^{\flatBernoulli}$, that is, where the $x$-coordinate is negative, $L^{\flatBernoulli}$ is a deterministic saw-tooth curve. It is clear that $\ell_N(s)$ for $s \leq 0$ satisfy inequality \eqref{eq:upper_bound_l}, and the ``flat'' part of $L^{\flatBernoulli}$ satisfies \eqref{eq:region_D} and \eqref{eq:L_N_outside}. On the other hand, $\ell_N(s)$ for $s > 0$ is defined by the simple symmetric random walk in \eqref{eq:corresp_RW}. It is well known that the path of a simple symmetric random walk is bounded by a parabola with probability close to $1$ provided that the parabola is high enough. Thus if $C = C_{\epsilon}$ is large enough, with probability $> 1 - \epsilon/2$, $\ell_N(s) < C + c_1 s^2$ for all $N$ large enough and then \eqref{eq:upper_bound_l} holds. Since $L^{\flatBernoulli}$ is also defined by the simple symmetric random walk, similarly \eqref{eq:region_D} and \eqref{eq:L_N_outside} are satisfied with probability $1 - \epsilon/2$ if $N$ is large enough. Thus we prove the claim.

By Theorem \ref{thm:main_TASEP}, we have that if the coefficients $C, c_1, c_2, c^*_3$ are chosen as above, then for large enough $N$
\begin{equation}
  \left\lvert \Prob \left( \frac{h^{\flatBernoulli}(2\sigma \cnought N^{\frac{2}{3}}; \anought^* N) - 2N}{2\dnought^* N^{\frac{1}{3}}} > -x \right) - \Prob \left( \max_{s \in \realR} \AiryA(s) - (s - \sigma)^2 + \ell_N(s) < x \right) \right\rvert < \epsilon,
\end{equation}
where $\AiryA(s)$ is an Airy process.

It is clear that as $N \to \infty$, on $(-\infty, 0]$, $\ell_N(s)$ uniformly converges to the constant function $0$. On the other hand, for positive $s$, by the correspondence \eqref{eq:corresp_RW} and Donsker's theorem, $\ell_N(s)$ weakly converges to $\sqrt{2} q^{-1/4}\B(s)$, where $\B(s)$ is a standard Brownian motion and the constant factor is the ratio
\begin{equation}
  \sqrt{2} q^{-1/4} = \frac{\sqrt{2\cnought N^{2/3}}}{2\dnought^* N^{1/3}}
\end{equation}
where $\cnought$ is defined in \eqref{eq:defn_nought} and $\dnought^*$ is defined in \eqref{eq:defn_noughtstar}. By an argument like that between \eqref{eq:cor_by_thm} and \eqref{eq:final_step_BM} in the proof of Corollary \ref{cor:corrollaries_of_inhomogeneous_LPP}\ref{thm:spiked_border}, we prove part \ref{enu:thm:Bernoulli_initial_condition:d} of Corollary \ref{thm:Bernoulli_initial_condition}.

\section{Proof of Theorem~\ref{thm:uniform_slow_decorr}} \label{sec:uniform_slow_decor}

Let $C'_N=[(C+1)N^\alpha]/N^\alpha$ which depends on $N$ and lies in the interval $[C,C+1]$. Define
 \begin{equation}
 H_{N,\pm}(s):=\frac{1}{\bnought N^{1/3}}\left(\G\left(N\pm 2C'_NN^\alpha+s\cnought N^{2/3}, N\pm 2C'_NN^\alpha-s \cnought N^{2/3}\right)-\anought\left(N\pm 2C'_NN^\alpha\right)\right)
 \end{equation}
 for all $s\in[-M,M]$. It is a direct to check that
 \begin{equation}
 \label{eq:relation_H_pm_and_H}
 H_{N,\pm}(s)=\left(1+2C'_NN^{\alpha-1}\right)^{-1/3}H_{N\pm 2C'_NN^\alpha}\left(s+O(N^{\alpha-1})\right),
 \end{equation}
 where the term $O(N^{\alpha-1})$ is independent of $s$.

We first prove the following claim.
\begin{claim}
\label{claim:comparison_H_pm_and_H}
For any given $\epsilon, \delta>0$, there exists a constant $N_1$ which only depends on $M,\alpha$ and $C$ such that
\begin{equation}
 \label{eq:aux_estimate_H}
 \Prob\left(\max_{s\in[-M,M]}|H_{N,\pm}(s)-H_N(s)|\ge\frac{\delta}{2}\right)<\frac{\epsilon}{2}
 \end{equation}
 for all $N>N_1$.
\end{claim}

To see this we first note that $H_N(s)$ is tight (see \cite[Lemma 5.3.]{Johansson03}), \ie, there exist constant $\delta'>0$ and $N'_1>0$ which only depend on $M, \epsilon$ and $\delta$ such that
\begin{equation}
\label{eq:tight_H_N}
\Prob\left(\max_{|s_1|,|s_2|\le M, |s_1-s_2|\le \delta'}\lvert H_N(s_1)-H_N(s_2)\rvert\ge\frac{\delta}{6}\right)<\frac{\epsilon}{6}
\end{equation}
for all $N\ge N'_1$.

The relation~\eqref{eq:relation_H_pm_and_H} implies that $H_{N,\pm}(s)$ are also tight. Therefore there exist constant $\delta''>0$ and $N''_1>0$ which only depend on $M,\epsilon$ and $\delta$ such that
\begin{equation}
\label{eq:tight_H_pm_N}
\Prob\left(\max_{|s_1|,|s_2|\le M, |s_1-s_2|\le \delta''}\lvert H_{N,\pm}(s_1)-H_{N,\pm}(s_2)\rvert\ge\frac{\delta}{6}\right)<\frac{\epsilon}{6}
\end{equation}
for all $N\ge N''_1$.

Now we fix $\delta'$ and $\delta''$, and denote $t_j=j\cdot\min\{\delta',\delta''\}$ for all integers $j$ such that $-M\le t_j\le M$. By the slow decorrelation of LPP (see, \cite[Theorem 2.1]{Corwin-Ferrari-Peche12}), we know that there exists some constant $N'''_1$ which depends on $C, \epsilon, \delta, \delta'$ and $\delta''$ such that
\begin{equation}
\label{eq:slow_decorrelation_pointwise}
\Prob\left(\max_{-M\le|j|\cdot\min\{\delta',\delta''\}\le M}\lvert H_{N,\pm}(t_j)-H_N(t_j)\rvert\ge \frac{\delta}{6}\right)< \frac{\epsilon}{6}
\end{equation}
for all $N\ge N'''_1$.

Note that for all $s\in[-M,M]$, there exists some $j$ such that $|t_j-s|\le \min\{\delta',\delta''\}$, and that
\begin{equation}
\lvert H_{N,\pm}(s)-H_N(s)\rvert\le \lvert H_{N,\pm}(t_j)-H_N(t_j)\rvert+\lvert H_{N,\pm}(s)-H_{N,\pm}(t_j)\rvert+\lvert H_{N}(s)-H_N(t_j)\rvert.
\end{equation}
Together with~\eqref{eq:tight_H_N},~\eqref{eq:tight_H_pm_N} and~\eqref{eq:slow_decorrelation_pointwise} we obtain Claim~\ref{claim:comparison_H_pm_and_H}.

Now we prove Theorem~\ref{thm:uniform_slow_decorr}. Note that for $s\in [-M,M]$ such that $s\cnought N^{2/3}\in\intZ$ we have
 \begin{equation}
 \begin{split}
 &\G\left(N+ 2C'_NN^\alpha+s\cnought N^{2/3}, N+ 2C'_NN^\alpha-s\cnought N^{2/3}\right)\\
 &-\G\left(N+l_N\left(s\right)N^\alpha+s\cnought N^{2/3}, N+l_N\left(s\right)N^\alpha-s\cnought N^{2/3}\right)\\
 &\geq G_{\left(N+ 2C'_NN^\alpha+s\cnought N^{2/3}, N+ 2C'_NN^\alpha-s\cnought N^{2/3}\right)}\left(N+l_N\left(s\right)N^\alpha+s\cnought N^{2/3}, N+l_N\left(s\right)N^\alpha-s\cnought N^{2/3}\right)
 \end{split}
 \end{equation}
 which has the same distribution as $\G\left(\left(2C'_N-l_N(s)\right)N^\alpha,\left(2C'_N-l_N(s)\right)N^\alpha\right)$. If $\alpha>1/3$, by applying Proposition~\ref{prop:previous_pt_to_pt}\ref{enu:prop:previous_pt_to_pt_2} we obtain the following estimate
 \begin{equation}
 \Prob\left(H_{N,+}(s)-\tilde H_N(s)\le-\frac{\delta}{2}\right)\le  e^{-c'N^{1-\alpha}}
 \end{equation}
 for all $N\ge N'_2$,
 where $c'$ and $N'_2$ are positive parameters independent of $s$. If $\alpha\le 1/3$, we have
 \begin{equation}
 \label{eq:estimate_difference_H_pm_and_H}
  \Prob\left(H_{N,+}(s)-\tilde H_N(s)\le-\frac{\delta}{2}\right)\le  \Prob\left(H_{N,+}(s)-\tilde H_N(s)\le-\frac{\delta}{2}N^{(3\alpha-2)/6}\right).
 \end{equation}
By applying Proposition~\ref{prop:previous_pt_to_pt}\ref{enu:prop:previous_pt_to_pt_2} again, we obtain
  \begin{equation}
 \Prob\left(H_{N,+}(s)-\tilde H_N(s)\le-\frac{\delta}{2}N^{(3\alpha-2)/6}\right) \le e^{-c''N^{\alpha/2}},
 \end{equation}
 for all $N\ge N''_2$, where $c''$ and $N''_2$  are positive  parameters independent of $s$. Therefore we still have the estimate~\eqref{eq:estimate_difference_H_pm_and_H} with $c'$ and $N'_2$ replaced by $c''$ and $N''_2$. By combining the above two cases we have
 \begin{equation}
 \Prob\left(\max_{s\in [-M,M], s\cnought N^{2/3}\in\intZ}\left(\tilde H_N(s)-H_{N,+}(s)\right)\ge\frac{\delta}{2}\right)\le \sum_{s\in [-M,M], s\cnought N^{2/3}\in\intZ} e^{-c'''N^{\min\{1-\alpha,\alpha/2\}}}
 \end{equation}
for all $N\ge N'''_2=\max\{N'_2,N''_2\}$, where $c'''=\min\{c',c''\}$. Note that the above estimate includes all the lattice points on the path $\{(N+s\cnought N^{2/3},N-s\cnought N^{2/3})\mid s\in[-M,M]\}$. Similarly one can obtain an analogous estimate including all the lattice points on the path $\{(l_N(s)N^\alpha+s\cnought N^{2/3},l_N(s)N^\alpha-s\cnought N^{2/3})\mid s\in[-M,M]\}$.  Moreover, the right hand side of the estimate tends to zero as $N\to\infty$ since there are only $o(N)$ terms in the summation. As a result, there exists an integer $N_2$ which depends on $M, C,\epsilon$, and $\delta$ such that
\begin{equation}
 \Prob\left(\max_{s\in [-M,M],\atop\mbox{lattice points}}\left(\tilde H_N(s)-H_{N,+}(s)\right)\ge\frac{\delta}{2}\right)<\frac{\epsilon}{2}
\end{equation}
for all $N\ge N_2$, where the maximum is taken over all the $s\in[-M,M]$ such that $(s\cnought N^{2/3},-s\cnought N^{2/3})$ or $(l_N(s)N^\alpha+ s\cnought N^{2/3},l_N(s)N^\alpha- s\cnought N^{2/3})$ is a lattice point. One can remove this restriction by
 using the definition of $H_N$ and $H_{N,+}$, and replacing the value of $\G$ at an arbitrary point by the interpolation of that on two nearby lattice points. Therefore there exists an integer $N_2$ which depends on $M, C,\epsilon$ such that
\begin{equation}
 \Prob\left(\max_{s\in [-M,M]}\left(\tilde H_N(s)-H_{N,+}(s)\right)\ge\frac{\delta}{2}\right)<\frac{\epsilon}{2}
\end{equation}
for all $N\ge N_2$.

By combing this estimate and Claim \ref{claim:comparison_H_pm_and_H}, we immediately have
\begin{equation}
\begin{split}
&\Prob\left(\max_{s\in[-M,M]}\left(\tilde H_N(s)-H_N(s)\right)\ge \delta\right)\\
&\le\Prob\left(\max_{s\in [-M,M]}\left(\tilde H_N(s)-H_{N,+}(s)\right)\ge\frac{\delta}{2}\right)+\Prob\left(\max_{s\in [-M,M]}\left(H_{N,+}(s)-H_N(s)\right)\ge\frac{\delta}{2}\right)\\
&<\epsilon
\end{split}
\end{equation}
for all $N\ge\max\{N_1,N_2\}$.

Similarly, there exists an integer $N_3$ which depends on $M, C,\epsilon$ and $\delta$ such that
\begin{equation}
 \Prob\left(\max_{s\in [-M,M]}\left(H_{N,-}(s)-\tilde H_N(s)\right)>\frac{\delta}{2}\right)<\frac{\epsilon}{2}
\end{equation}
for all $N\ge N_3$. By combing this estimate and Claim \ref{claim:comparison_H_pm_and_H}, we have
\begin{equation}
\begin{split}
&\Prob\left(\max_{s\in[-M,M]}\left( H_N(s)-\tilde H_N(s)\right)\ge \delta\right)\\
&\le\Prob\left(\max_{s\in [-M,M]}\left(H_N(s)-H_{N,-}(s)\right)\ge\frac{\delta}{2}\right)+\Prob\left(\max_{s\in [-M,M]}\left(H_{N,-}(s)-\tilde H_N(s)\right)\ge\frac{\delta}{2}\right)\\
&<\epsilon
\end{split}
\end{equation}
for all $N\ge\max\{N_1,N_2\}$. Theorem~\ref{thm:uniform_slow_decorr} follows immediately by taking $N_0=\max\{N_1,N_2,N_3\}$.

\section{Gibbs property of multi-layer discrete PNG and proof of Lemma \ref{lem:PNG_max}} \label{sec:Gibbs}

The goal of this section is to prove Lemma \ref{lem:PNG_max}. The proof relies on the correspondence between the LPP model and the multi-layer discrete polynuclear growth (PNG) model. The essential ingredient of the proof is the Gibbs property of the multi-layer discrete PNG model, analogous to the Gibbs property of the nonintersecting Brownian motions studied in \cite{Corwin-Hammond11}. We describe the multi-layer discrete PNG model and its relation to LPP, following closely to the presentation in \cite{Johansson03}, to facilitate the proof. Then we prove Lemma \ref{lem:PNG_max} based on technical results in Lemmas \ref{lem:monotonicity_of_random_walk} and \ref{lem:middle_pt}. The strategy of our proof is similar to that of \cite[Lemma 5.1]{Corwin-Hammond11}.


Let $I$ be an interval and $h(t)$ a function defined on $I$ that satisfies
\begin{equation}
  h(t) = h([t]) \in \intZ, \quad \text{and} \quad h(2m) \geq h(2m - \epsilon),\ h(2m + 1) \leq h(2m + 1 - \epsilon), \quad \text{for $m \in \intZ \cap I$},
\end{equation}
then we say that $h(t)$ is a PNG trajectory line on $I$.
If two PNG trajectory lines $h(t)$ and $g(t)$ on the same interval $I$ satisfy
\begin{equation}
  \limsup_{t \to t_0} g(t_0) < \liminf_{t \to t_0} h(t_0) \quad \text{for all $t_0 \in I$},
\end{equation}
we say that $h(\cdot)>g(\cdot)$ on $I$.

Fix a parameter $N\in \intZ_+$ and a constant $\epsilon\in (0,1)$. The multi-layer discrete PNG model is defined by an ensemble of infinitely many strictly ordered  PNG  trajectory lines $h_0,h_1,\cdots$ on $[-2N,2N]$. We say $h_0,h_1,\ldots$ form an $N$-permissible configuration if they satisfy the initial and terminal conditions
\begin{equation} \label{eq:init_term_conditions_PNG}
  h_i(-2N + 1 - \epsilon) = -i, \quad h_i(2N - 1 + \epsilon) = -i, \quad i = 0, 1, 2, \dotsc,
\end{equation}
and also satisfy the inequalities $h_i(\cdot)>h_{i+1}(\cdot)$ on $[-2N,2N]$ for all $i=0,1,\ldots$. One example of such a configuration is given in Figure \ref{fig:PNG}. Note that necessarily, for $i\geq N$, $h_i(\cdot)\equiv -i$ on $[-2N,2N]$.

\begin{figure}[htb]
  \centering
  \includegraphics{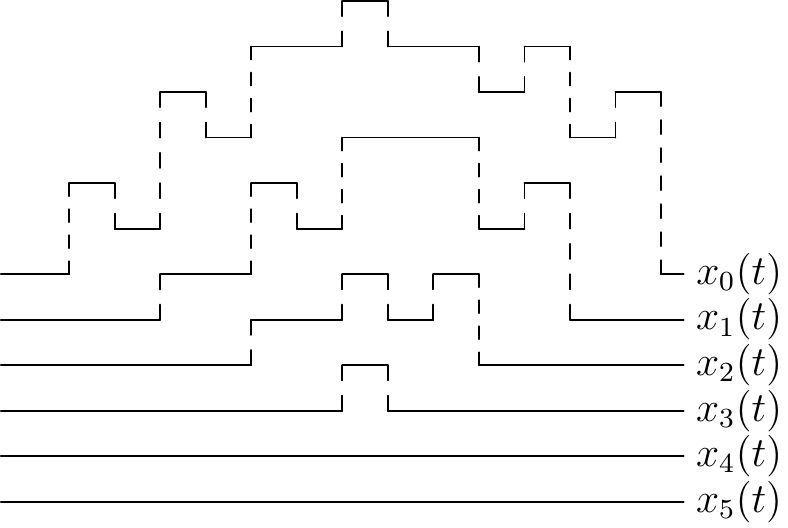}
  \caption{An example of multi-layer discrete PNG with $N = 4$.}
  \label{fig:PNG}
\end{figure}

 We define a weight $w$ for a PNG trajectory line $h$ on an interval $I = [a, b]$ as
\begin{equation} \label{eq:weight_PNG_trajectory}
  w(h) = \prod^{[b]}_{k = [a] + 1} p(\lvert h(k) - h(k - \epsilon) \rvert) \quad \text{where} \quad p(k) = \sqrt{1 - q} (\sqrt{q})^k,
\end{equation}
and define the weight for an $N$-permissible configuration of PNG trajectory lines $(h_0, h_1, \dotsc)$
\begin{equation} \label{eq:weight_of_2N_tuple}
  w(h_0, h_1, \dotsc) = \prod^{N - 1}_{k = 0} w(h_i),
\end{equation}
where $I = [-2N + 1 - \epsilon, 2N - 1 + \epsilon]$ in the formula of $w(h_i)$.  This product is restricted to $k\leq N-1$ since all other lines are constant as observed earlier.  The normalization $\sqrt{1 - q}$ is chosen such that
\begin{equation}
  \sum_{\text{all $N$-permissible configurations}} w(h_0, h_1, \dotsc) = 1.
\end{equation}
That this is the case can be shown from  \cite[Claim 3.10 and Proposition 3.11]{Johansson03}. This implies that the weight \eqref{eq:weight_of_2N_tuple} defines a probability on the set of all $N$-permissible configurations.

Furthermore, \cite[Proposition 3.11]{Johansson03} implies that for any $N$, the joint distribution of $\G(N + k, N - k)$ for $k = -N, -N + 1, \dotsc, N$, as defined in \eqref{eq:defn_inverse} is the same as the joint distribution of $h_0(2k)$ for $k = -N, -N + 1, \dotsc, N$, if $(h_0(t), h_1(t) \dotsc)$ is a random $N$-permissible configuration with probability given in \eqref{eq:weight_of_2N_tuple}. Then the point-to-curve LPP in Lemma \ref{lem:PNG_max} is expressed as \cite[Proposition 3.11]{Johansson03}
\begin{multline} \label{LPP_PNG_correspondence}
  \max_{K_1 \leq s \leq K_2 - c(K_2 - K_1)} G(N + l^0(s) + s, N + l^0(s) - s) = \max_{K_1 \leq k \leq K_2 - c(K_2 - K_1)} G(N + k, N - k) \\
  \stackrel{\mathrm{d}}{=} \max_{K_1 \leq k \leq K_2 - c(K_2 - K_1)} h_0(2k) = \max_{t \in [K_1, K_2 - c(K_2 - K_1)]} h_0(t).
\end{multline}
The proof of Lemma \ref{lem:PNG_max} relies on the Gibbs property of the probability space of permissible $2N$-tuples, in particular the Gibbs property as follows.
\begin{lem} \label{lem:Gibbs_random_walk}
   Consider $t_1 < t_2$, with $t_1,t_2\in (-2N + 1 - \epsilon, 2N - 1 + \epsilon)$ and consider $\tilde{h}(\cdot) = \big(\tilde{h}_0(\cdot),\tilde{h}_1(\cdot),\ldots\big)$ distributed according to the multi-layer discrete PNG model. Then the law of $\tilde{h}_0$ restricted to the interval $[t_1,t_2]$ is distributed according to the PNG trajectory of a single line $h(\cdot)$ on the interval $[t_1,t_2]$ conditioned on $h(t_1)=\tilde{h}_0(t_1)$, $h(t_2)=\tilde{h}_0(t_2)$, and $h(\cdot)>\tilde{h}_1(\cdot)$ on the entire interval.
\end{lem}
\begin{proof}
  This lemma is a direct consequence of the formulas \eqref{eq:weight_PNG_trajectory} and \eqref{eq:weight_of_2N_tuple} that define the probability distribution of PNG trajectory lines and $N$-permissible configurations.
\end{proof}
We need two more lemmas. The first is a monotone coupling result:
\begin{lem} \label{lem:monotonicity_of_random_walk}
  Let $t_1 < t_2 < t_3 \in \realR$, $a_1, a_2, a_3 \in \intZ$ and $\tilde{h}(t)$ be a fixed PNG trajectory line on $[t_1, t_3]$ such that $\tilde{h}(t_1) < a_1$, $\tilde{h}(t_3) < a_3$. Suppose $h(t)$ is a random variable in the space of PNG trajectory lines $H := \big\{ h(t) \text{ on }[t_1, t_3] \mid h(t_1) = a_1,\ h(t_3) = a_3, \text{ and } h(\cdot)>\tilde{h}(\cdot) \big\}$ where the probability is given by the weight $w(h)$ as in \eqref{eq:weight_PNG_trajectory} up to a normalization constant, and suppose $g(t)$ is a random variable in the space of PNG trajectory lines $G := \{ g(t) \text{ on $[t_1, t_3]$} \mid g(t_1) = a_1,\ g(t_3) = a_3 \}$ where the probability is also given by the weight $w(g)$ as in \eqref{eq:weight_PNG_trajectory} up to a normalization constant. Then it follows that

  \begin{equation} \label{eq:monotonicity_lem}
    \Prob(h(t_2) \geq a_2) \geq \Prob(g(t_2) \geq a_2).
  \end{equation}
\end{lem}
\begin{proof}[Sketch of proof]
  In the proof of \cite[Lemma 2.6]{Corwin-Hammond11}, the result of this lemma is shown to hold if the PNG trajectory line is replaced by the trajectory of a standard random walk. The same method, namely the coupling of Monte-Carlo Markov chains, works in our situation.

  We consider a continuous-time Markov chain dynamic on the countable sets $H$ and $G$. Without loss of generality, we assume that $t_1$ and $t_3$ are even integers. To distinguish the time variable of the Markov chain dynamic and the variables of $h(t)$ and $g(t)$, we denote the Markov time as $\tau$, and write the random PNG trajectory lines as $h_{\tau}(t)$ and $g_{\tau}(t)$ respectively. The time $0$ configuration of $h_0(t)$ is chosen arbitrarily in $H$ and we let $g_0(t) = h_0(t)$. The dynamics of the Markov chain are as follows. For each integer $t_0 \in \{ t_1 + 1, t_1 + 2, \dotsc, t_3 - 1 \}$, there is an independent exponential clock which rings at rate $1$. For each $\tau > 0$, let $r(\tau)$ be \iid\ random variables with geometric distribution such that $\Prob(r(\tau) = k) = (1 - q) q^k$ for $k = 0, 1, 2, \dotsc$. When the clock labeled by $t_0$ rings, the random PNG trajectory line $h_{\tau}(t)$ remains the same for $t \notin [t_0, t_0 + 1)$, and changes the value on $[t_0, t_0 + 1)$ into (1) $\max(h_{\tau}(t_0 - 1), h_{\tau}(t_0 + 1) + r(\tau)$ if $t_0$ is even, or (2) $\min(h_{\tau}(t_0 - 1), h_{\tau}(t_0 + 1)) - r(\tau)$ if $t_0$ is odd. Likewise, according to the same clock, the random PNG trajectory line $g_{\tau}(t)$ remains the same for $t \notin [t_0, t_0 + 1)$ and changes the value on $[t_0, t_0 + 1)$ into (1) $\max(g_{\tau}(t_0 - 1), g_{\tau}(t_0 + 1) + r(\tau)$ if $t_0$ is even, or (2a) $\min(g_{\tau}(t_0 - 1), g_{\tau}(t_0 + 1)) - r(\tau)$ if $t_0$ is odd and $\min(g_{\tau}(t_0 - 1), g_{\tau}(t_0 + 1)) - r(\tau) > \max(\tilde{h}(t_0 - 1), \tilde{h}(t_0 + 1))$, or (2b) remains the same otherwise.

  Then we observe that for any $\tau > 0$, $h_{\tau}(t) \geq g_{\tau}(t)$ for all $t \in [t_1, t_3]$. Another fact is that the marginal distributions of these time dynamics converge to the invariant measures for this Markov chain, which are given by the weight function \eqref{eq:weight_PNG_trajectory} on the state spaces $G$ and $H$ respectively. This can be confirmed by checking that the multi-layer PNG model measure is the unique invariant measure under these irreducible, aperiodic  Markov dynamics.
\end{proof}

\begin{lem} \label{lem:middle_pt}
  Let $t_1 < t_2 < t_3 \in \realR$,  $a_1, a_3 \in \intZ$ and $a_2 \in \realR$ such that $(t_1, a_1), (t_2, a_2), (t_3, a_3)$ are collinear, \ie,
  \begin{equation}
    \frac{a_2 - a_1}{t_2 - t_1} = \frac{a_3 - a_2}{t_3 - t_2}.
  \end{equation}
  Let $g(t)$ be a random variable in the space of PNG trajectory lines with fixed ends $G := \{ g(t) \text{ on $[t_1, t_3]$} \mid g(t_1) = a_1,\ g(t_3) = a_3 \}$ where the probability is given by the weight $w(g)$ as in \eqref{eq:weight_PNG_trajectory} up to a normalization constant. Then
  \begin{equation}
    \Prob(g(t_2) \geq a_2) \geq \frac{1}{2} - \delta_{\min(t_2 - t_1, t_3 - t_2)},
  \end{equation}
  where for any $t > 0$, $\delta_t > 0$ is a decreasing function in $t$ and $\delta_t \to 0$ as $t \to \infty$.
\end{lem}
\begin{proof}
  Without loss of generality, we assume that $t_1 = a_1 = 0$ and then $a_2 = (t_2/t_3)a_3$. We also assume in the proof that $t_1, t_2, t_3$ are even integers. Consider the \iid\ discrete random variables $X_1, X_2, \dotsc$ with support $\intZ$ and distribution
  \begin{equation}
    \Prob(X_1 = k) = \frac{1 - \sqrt{q}}{1 + \sqrt{q}} (\sqrt{q})^{\lvert k \rvert}, \quad k = 0, \pm 1, \pm 2, \dotsc,
  \end{equation}
  and define $S_n = \sum^n_{k = 1} X_k$. Then the distribution of $g(t_2)$ is the same as the distribution of $S_{t_2/2}$ under the condition that $S_{t_3/2} = a_3$. We take a change of measure, and define another sequence of \iid\ discrete random variables $X'_1, X'_2, \dotsc$ with support $\intZ - a_3/t_3$ and distribution
  \begin{equation}
    \Prob \left( X'_1 = k - \frac{a_3}{t_3} \right) = \frac{(1 - p\sqrt{q})(1 - \sqrt{q}/p)}{1 - q} \times
    \begin{cases}
      (p\sqrt{q})^k & \text{if $k \geq 0$}, \\
      (\sqrt{q}/p)^k & \text{if $k < 0$},
    \end{cases}
  \end{equation}
  where $p$ is the real number in $(\sqrt{q}, \sqrt{q}^{-1})$ that satisfies
  \begin{equation} \label{eq:equation_for_p_change_of_measure}
    \frac{(p^2 - 1)\sqrt{q}}{(1 - p\sqrt{q})(p - \sqrt{q})} = \frac{a_3}{t_3}.
  \end{equation}
  Then if we define $S'_n = \sum^n_{i = 1} X'_i$, the distribution of $g(t_2) - a_2$ is the distribution of $S'_{t_2/2}$ under the condition that $S'_{t_3/2} = 0$. Explicit computation shows that the mean of $X'_1$ is zero and the variance of $X'_1$ is bounded below by a positive constant independent of $a_3/t_3$. Thus the random walk with increment $X'_k$ conditioned with $S'_{t_3/2} = 0$ converges weakly to a Brownian motion as $t_3/2 \to \infty$, and the convergence is uniform in $a_3/t_3$. Since for a Brownian bridge from $0$ to $0$, at any time between the initial and the terminal times, the probability that the position of particle is positive equals $1/2$, we have that the probability that $g(t_2) - a_2$ is positive converges to $1/2$ as the total steps of the random walk $t_3/2 \to \infty$ and both $t_2/2 \to \infty$ and $(t_3 - t_2)/2 \to \infty$. Since the convergence of the conditioned random walk to a Brownian bridge is uniform in $a_3/t_3$, the convergence of $\Prob(g(t_2) - a_2)$ to $1/2$ is also uniform in $a_3/t_3$. We thus prove the lemma.
\end{proof}

Now we can prove Lemma \ref{lem:PNG_max}. By \eqref{LPP_PNG_correspondence}, the lemma is transformed into a property of multi-layer discrete PNG model with parameter $N$. We denote $K'_2 = K_2 - c(K_2 - K_1)$, and let $\big(h_0(\cdot), h_1(\cdot),\ldots\big)$ be a multi-layer PNG model distributed ensemble of lines with probability defined by \eqref{eq:weight_of_2N_tuple}. Then we have
\begin{equation} \label{eq:triple_sum_of_probab}
  \begin{split}
    & \Prob \left( \max_{K_1 \leq k \leq K'_2} h_0(2k) \geq M_1 \right) \\
    \leq {}& \Prob\big(h_0(2K_3) < M_3\big) + \Prob \left( \max_{K_1 \leq k \leq K'_2} h_0(2k) \geq M_1 \text{ and } h_0(2K_3) \geq M_3 \right) \\
    \leq {}& \Prob\big(h_0(2K_3) < M_3\big) + \sum_{K_1 \leq K \leq K'_2} \sum^{\infty}_{M'_1 = M_1} \sum^{\infty}_{M'_3 = M_3} \Prob \left(
      \begin{gathered}
        \max_{K_1 \leq k < K} h_0(2k) < M_1, \\
        h_0(2K) = M'_1 \text{ and } h_0(2K_3) = M'_3
      \end{gathered}
      \right).
  \end{split}
\end{equation}

By Lemmas \ref{lem:Gibbs_random_walk} and \ref{lem:monotonicity_of_random_walk}, we have the inequality for the conditional probability
\begin{equation} \label{eq:conditional_prob_ineq}
   \Prob \left( h_0(2K_2) \geq M_2 \left\lvert
       \begin{gathered}
         \max_{K_1 \leq k < K} h_0(2k) < M_1, \\
         h_0(2K) = M'_1\text{ and } h_0(2K_3) = M'_3
       \end{gathered}
     \right. \right) \leq \Prob(g(2K_2) \geq M_2),
\end{equation}
where $g(t)$ is a random variable in the space of PNG trajectory lines with fixed ends $G := \{ g(t) \text{ on $[2K, 2K_3]$} \mid g(2K) = M'_1,\ g(2K_3) = M'_3 \}$ and the probability is given by the weight $w(g)$ as in \eqref{eq:weight_PNG_trajectory} up to a normalization constant.

Denote
\begin{equation}
  M'_2 = \frac{K_3 - K_2}{K_3 - K} M'_1 + \frac{K_2 - K}{K_3 - K} M'_3,
\end{equation}
such that $(K, M'_1), (K_2, M'_2), (K_3, M'_3)$ are collinear. It is clear that $M'_2 \geq M_2$, and then by Lemma \ref{lem:middle_pt}
\begin{equation} \label{eq:random_walk}
  \Prob \big(g(2K_2) \geq M_2\big) \geq \Prob \big(g(2K_2) \geq M'_2\big) > \frac{1}{2} - \delta_{\min(K_2 - K, K_3 - K_2)} > \frac{1}{2} - \delta_{\min(c(K_2 - K_1), K_3 - K_2)}.
\end{equation}
where $\delta_t$ is the same as in Lemma \ref{lem:middle_pt}.

Thus by \eqref{eq:conditional_prob_ineq} and \eqref{eq:random_walk},
\begin{multline}
  \Prob \left(
      \begin{gathered}
        \max_{K_1 \leq k < K} h_0(2k) < M_1, \\
        h_0(2K) = M'_1 \text{ and } h_0(2K_3) = M'_3
      \end{gathered}
      \right) < \\
      \frac{1}{\frac{1}{2} - \delta_{\min(c(K_2 - K_1), K_3 - K_2)}} \Prob \left(
      \begin{gathered}
        h_0(2K_2) \geq M_2,\ \max_{K_1 \leq k < K} h_0(2k) < M_1, \\
        h_0(2K) = M'_1 \text{ and } h_0(2K_3) = M'_3
      \end{gathered}
      \right),
\end{multline}
and then
\begin{equation} \label{eq:final_estimate}
  \begin{split}
    & \sum_{K_1 \leq K \leq K'_2} \sum^{\infty}_{M'_1 = M_1} \sum^{\infty}_{M'_3 = M_3} \Prob \left(
      \begin{gathered}
        \max_{K_1 \leq k < K} h_0(2k) < M_1, \\
        h_0(2K) = M'_1 \text{ and } h_0(2K_3) = M'_3
      \end{gathered}
    \right) \\
    < {}& \frac{1}{\frac{1}{2} - \delta_{\min(c(K_2 - K_1), K_3 - K_2)}} \sum_{K_1 \leq K \leq K'_2} \sum^{\infty}_{M'_1 = M_1} \sum^{\infty}_{M'_3 = M_3} \Prob \left(
      \begin{gathered}
        h_0(2K_2) \geq M_2,\ \max_{K_1 \leq k < K} h_0(2k) < M_1, \\
        h_0(2K) = M'_1 \text{ and } h_0(2K_3) = M'_3
      \end{gathered}
    \right) \\
    \leq {}& \frac{1}{\frac{1}{2} - \delta_{\min(c(K_2 - K_1), K_3 - K_2)}} \Prob(h_0(2K_2) \geq M_2).
  \end{split}
\end{equation}
Substitute \eqref{eq:final_estimate} into \eqref{eq:triple_sum_of_probab} and use the correspondence \eqref{LPP_PNG_correspondence}, we obtain the proof of Lemma \ref{lem:PNG_max} with the $\epsilon_t$ there determined by $2 + \epsilon_t = (\frac{1}{2} - \delta_t)^{-1}$ where $\delta_t$ is that in Lemma \ref{lem:middle_pt}.

\appendix

\section{Proof of Lemma \ref{lem:weaker_technical}} \label{sec:appendix}

In this appendix we prove the following estimate of $G([\gamma N],N)$:
\begin{lem}
For any fixed $\gamma_0>1$, there exist some constant $L>0$ and $\delta>0$ such that
\begin{equation}
\label{eq:lemma_intermediate_estimate_G}
\Prob\left(G([\gamma N], N)\ge a_0(\gamma) N+sb_0(\gamma) N^{1/3}\right)\le e^{-cs^{3/2}},
\end{equation}
for large $N$ and all $\gamma\in[\gamma_0^{-1},\gamma_0]$, $s\in[L,\delta N^{2/3}]$. Here $a_0(\gamma)$ and $b_0(\gamma)$ are defined in~\eqref{eq:LLN_result} and~\eqref{eq:fluctuation_constant_def}, $c>0$ is a constant which only depends on $\gamma_0, L$ and $\delta$.
\end{lem}
  \begin{proof}

    The following formula for the distribution of $G(M,N)$ was known \cite{Baik-Rains01b}
    \begin{equation}
    \label{eq:distribution_G_Toeplitz}
      \Prob\left(G(M,N)\le n\right)=(1-q)^{MN}D_n(\phi),
    \end{equation}
    where $\phi(z):=(1+\sqrt{q}z)^M(1+\frac{\sqrt{q}}{z})^N$, and $D_n(\phi)$ is the $n$-th Toeplitz determinant with symbol $\phi$:
    \begin{equation}
    D_n(\phi):=\det\left(\int_{|z|=1}z^{-j+k}\phi(z)\frac{dz}{2\pi iz}\right)_{j,k=0}^{n-1}.
    \end{equation}
   Note that one can take $n\to\infty$ in~\eqref{eq:distribution_G_Toeplitz} and obtain
   \begin{equation}
   D_\infty(\phi):=\lim_{n\to\infty}D_n(\phi)=(1-q)^{-MN}.
   \end{equation}

   Now we apply the Geronimo-Case-Borodin-Okounkov formula \cite{Case-Geronimo79, Borodin-Okounkov00} and obtain
    \begin{equation}
      \Prob\left(G(M,N)\le n\right)=D_\infty(\phi)^{-1}D_n(\phi)=\det(1-K_n),
    \end{equation}
    where $K_n$ is an operator on $l^2\{n,n+1,\cdots\}$ with kernel
    \begin{equation}
      K_n(i,j)=\sum_{k=1}^{\infty}U(i,k)V(k,j).
    \end{equation}
    Here
    \begin{equation}
    \begin{split}
      U(i,k):=\int_{|z|=1}\left(1-\frac{\sqrt{q}}{z}\right)^N\left(1-\sqrt{q}z\right)^{-M}z^{-i-k}\frac{dz}{2\pi iz},\\
      V(k,j):=\int_{|z|=1}\left(1-\frac{\sqrt{q}}{z}\right)^{-N}\left(1-\sqrt{q}z\right)^{M}z^{j+k}\frac{dz}{2\pi iz}.
      \end{split}
    \end{equation}

    Now we consider the asymptotics of $\det(1-K_n)$ when $M=[\gamma N]$, $n=a_0(\gamma)N+sb_0(\gamma)N^{1/3}$ and $N\to\infty$. Here $\gamma \in [\gamma_0^{-1},\gamma_0]$ and $s\in [L,\delta N^{2/3}]$ for some parameters $L>0$ and $\delta>0$.

    Let
    \begin{equation}
      z_0:=\frac{1+\sqrt{\gamma q}}{\sqrt{\gamma}+\sqrt{q}}.
    \end{equation}
Note that if we replace the kernels $U$ and $V$ by the following $\tilde U$ and $\tilde V$, the determinant $\det(1-K_n)$ does not change.
\begin{equation}
\begin{split}
\tilde U(i,k):= \left(1-\frac{\sqrt{q}}{z_0}\right)^{-N}\left(1-\sqrt{q}z_0\right)^Mz_0^{i+k}U(i,k),\\
\tilde V(k,j):=V(k,j)\left(1-\frac{\sqrt{q}}{z_0}\right)^N\left(1-\sqrt{q}z_0\right)^{-M}z_0^{-j-k}.
\end{split}
\end{equation}

    Write $i=a_0(\gamma)N+xb_0(\gamma)N^{1/3}$, $j=a_0(\gamma)N+yb_0(\gamma)N^{1/3}$ and $k=ub_0(\gamma)N^{1/3}$, where $x,y\ge s,$ and $u\ge 0$. Then we have
    \begin{equation}
      \tilde U(i,k)=e^{Nf(z_0)}\int_{|z|=1} e^{\left(-Nf(z)+N^{1/3}\phi(z)\right)}\frac{dz}{2\pi iz}
    \end{equation}
    where
    \begin{equation}
    f(z)=-\log(1-\frac{\sqrt{q}}{z})+\gamma\log(1-\sqrt{q}z)+a_0(\gamma)\log z,
    \end{equation}
    and $\phi(z)=-(x+u)b_0(\gamma)\log (z/z_0)$.

    Note that
    \begin{equation}
      f'(z)=-\frac{\sqrt{q}}{1-q}\cdot\frac{\left((\sqrt{\gamma}+\sqrt{q})z-(1+\sqrt{\gamma q})\right)^2}{z(z-\sqrt{q})(1-\sqrt{q}z)}.
    \end{equation}
    Therefore near $z_0$, we have the following expansions
    \begin{equation}
    f(z)=f(z_0)-\frac{q^{1/2}(\sqrt{q}+\sqrt{\gamma})^{5}}{3\gamma^{1/2}(1-q)^3(1+\sqrt{q\gamma})}(z-z_0)^3+O(|z-z_0|^4),
    \end{equation}
    and
    \begin{equation}
    \phi(z)=-(x+u)\left(\frac{q^{1/6}(\sqrt{q}+\sqrt{\gamma})^{5/3}}{\gamma^{1/6}(1-q)(1+\sqrt{\gamma q})^{1/3}}(z-z_0)+O(|z-z_0|^2)\right).
    \end{equation}

We deform the contour such that it intersects a small neighborhood of $z_0$. For all $z$ on the contour but outside the above neighborhood of $z_0$, $\Re (f(z)-f(z_0))\ge c$ and $\Re \phi(z) \le - c(x+u)$ for some positive constant $c$. Thus by changing the variables near $z_0$ one can obtain
    \begin{equation}
    \begin{split}
      \tilde U(i,k)&=O(e^{-c\epsilon^3N})+b_0(\gamma)^{-1}N^{-1/3}\int_{2\epsilon N^{1/3}e^{-i\pi/3 }}^{2\epsilon N^{1/3} e^{i\pi/3}}e^{\frac{1}{3}\xi^3-(x+u)\xi}\frac{d\xi}{2\pi i}(1+O(N^{-1/3}))\\
      &=O(e^{-c\epsilon^3N})+b_0(\gamma)^{-1}N^{-1/3}\Ai(x+u)(1+O(N^{-1/3})),
      \end{split}
    \end{equation}
    where $c,\epsilon>0$ are constants which only depend on $L$ (the lower bound of $x+u$).  Similarly we have
    \begin{equation}
    \begin{split}
      \tilde V(k,j)=O(e^{-c\epsilon^3N})+b_0(\gamma)^{-1}N^{-1/3}\Ai(y+u)(1+O(N^{-1/3})).\\
    \end{split}
    \end{equation}
    Hence
    \begin{equation}
    \begin{split}
    b_0(\gamma)N^{1/3}K_n(i,j) &=O(e^{-c\epsilon^3N})+O(N^{1/3}e^{-c\epsilon^3N})\int_0^\infty\Ai(y+u)du+O(N^{1/3}e^{-c\epsilon^3N})\int_0^\infty\Ai(x+u)du
  \\
  &+\int_0^\infty\Ai(x+u)\Ai(y+u)du(1+O(N^{-1/3}))
    \end{split}
    \end{equation}

    Note that $x,y\ge s\ge L$. By using the asymptotics of the Airy function, we immediately obtain
    \begin{equation}
    |b_0(\gamma)N^{1/3}K_n(i,j)|\le e^{-c\rq{}(\min\{x^{3/2},c\rq{}\rq{}N\}+\min\{y^{3/2},c\rq{}\rq{}N\})}
    \end{equation}
    for large enough $N, L$, where $c\rq{}, c\rq{}\rq{}>0$ are both independent of $x,y,\gamma$.

    Therefore
$
    | \Tr(K_n^l)|\le e^{-c\rq{}ls^{3/2}}, l=1,2,\cdots,
$
    for large enough $N, L$, and $s\in [L, \delta N^{2/3}]$, provided $\delta^{3/2}\le c\rq{}\rq{}$. This estimate implies the following
    \begin{equation}
    \left|\sum_{k_i\in\{n,n+1,\cdots\}, i=1,\cdots, l}\frac{1}{l!}\det\left(K_n(k_i,k_j)\right)_{i,j=1}^l\right|\le e^{-c\rq{}ls^{3/2}}.
    \end{equation}

    Hence
    \begin{equation}
    \Prob\left(G([\gamma N], N)\ge a_0(\gamma) N+sb_0(\gamma) N^{1/3}\right)= 1-\det(1-K_n)\le \sum_{l=1}^{\infty}e^{-c\rq{}ls^{3/2}}
    \end{equation}
    and the lemma follows.
  \end{proof}

 \bibliographystyle{plain}
\bibliography{bibliography}

\end{document}